\documentclass{article}
\usepackage[utf8]{inputenc}
\usepackage{a4wide}
\usepackage{amsmath}
\usepackage{amsfonts}
\usepackage{appendix}
\usepackage{amsthm}
\usepackage{import}
\usepackage{libertine}
\usepackage[libertine]{newtxmath}
\usepackage{graphicx}
\usepackage[usenames]{color}
\usepackage{subfigure}
\usepackage{mathtools}
\usepackage{empheq}
 \usepackage{hyperref}
\usepackage{cases}
\usepackage{geometry}
\usepackage{dsfont}
\usepackage[inline]{enumitem}
\usepackage{cases}

\DeclareFontFamily{U}{mathx}{\hyphenchar\font45}
\DeclareFontShape{U}{mathx}{m}{n}{<->mathx10}{}
\DeclareFontSubstitution{U}{mathx}{m}{n}
\DeclareSymbolFont{mathx}{U}{mathx}{m}{n}

\DeclareMathSymbol{\intop}  {\mathop}{mathx}{"B3}
\DeclareMathSymbol{\iintop} {\mathop}{mathx}{"B4}
\DeclareMathSymbol{\iiintop}{\mathop}{mathx}{"B5}
\DeclareMathSymbol{\ointop} {\mathop}{mathx}{"B6}
\DeclareMathSymbol{\oiintop}{\mathop}{mathx}{"B7}

\makeatletter
\DeclareFontFamily{OMX}{MnSymbolE}{}
\DeclareSymbolFont{MnLargeSymbols}{OMX}{MnSymbolE}{m}{n}
\SetSymbolFont{MnLargeSymbols}{bold}{OMX}{MnSymbolE}{b}{n}
\DeclareFontShape{OMX}{MnSymbolE}{m}{n}{
	<-6>  MnSymbolE5
	<6-7>  MnSymbolE6
	<7-8>  MnSymbolE7
	<8-9>  MnSymbolE8
	<9-10> MnSymbolE9
	<10-12> MnSymbolE10
	<12->   MnSymbolE12
}{}
\DeclareFontShape{OMX}{MnSymbolE}{b}{n}{
	<-6>  MnSymbolE-Bold5
	<6-7>  MnSymbolE-Bold6
	<7-8>  MnSymbolE-Bold7
	<8-9>  MnSymbolE-Bold8
	<9-10> MnSymbolE-Bold9
	<10-12> MnSymbolE-Bold10
	<12->   MnSymbolE-Bold12
}{}
\let\llangle\@undefined
\let\rrangle\@undefined
\DeclareMathDelimiter{\llangle}{\mathopen}
{MnLargeSymbols}{'164}{MnLargeSymbols}{'164}
\DeclareMathDelimiter{\rrangle}{\mathclose}
{MnLargeSymbols}{'171}{MnLargeSymbols}{'171}
\makeatother

\definecolor{red}{rgb}{1.0,0.0,0.0}

\definecolor{blu}{rgb}{0.0,0.0,0.0}

\definecolor{gre}{rgb}{0.03,0.50,0.03}

\theoremstyle{plain}

\newtheorem{theorem}{Theorem}[section]
\newtheorem{lemma}[theorem]{Lemma}
\newtheorem{definition}[theorem]{Definition}

\newtheorem{proposition}[theorem]{Proposition}
\newtheorem{remark}[theorem]{Remark}
\newtheorem{assumption}[theorem]{Assumption}

\newcommand{\eps}{\varepsilon}

\def\R{\mathbb R}

\def\P{\mathbb P}

\def\calk{{\cal K}}

\def\calk{{\mathcal K}}

\def\to{\rightarrow}

\newcommand{\bE}{\mathbb{E}}

\newcommand{\bN}{\mathbb{N}\,}
\newcommand{\bP}{\mathbb{P}}

\newcommand{\bR}{\mathbb{R}}
\newcommand{\bS}{\mathbb{S}\,}

\newcommand{\rL}{\mathscr{L}}

\newcommand{\sF}{\mathcal{F}}

\newcommand{\sK}{\mathcal{K}}

\newcommand{\sP}{\mathcal{P}}
\newcommand{\sQ}{\mathcal{Q}}

\newcommand{\sW}{\mathcal{W}}

\newcommand{\ud}{\mathrm{d}}

\newcommand{\ind}{\mathds{1}}

\title{A mean field game model with non-local spatial interactions and resources accumulation}

 \author{
Daria Ghilli\footnote{Dipartimento di Scienze Economiche e Aziendali, Università di Pavia, Italy; email: daria.ghilli@unipv.it},
\;
Fausto Gozzi\footnote{Dipartimento di AI, Data and Decision Sciences,
LUISS University, Roma, Italy;
e-mail: fgozzi@luiss.it},
\;
Giovanni Zanco\footnote{Dipartimento di Ingegneria dell'Informazione e Scienze Matematiche,
Universit\`a di Siena, Italy;
e-mail: giovanni.zanco@unisi.it}
}

\begin{document}

\maketitle

\vspace{-1truecm}

\begin{abstract}
We study a family of mean field games arising in modeling the behavior of strategic economic agents which move across space maximizing their utility from consumption and have the possibility to accumulate resources for production (such as human capital).
The resulting mean field game PDE system is not covered in the actual literature on the topic as it displays weaker assumptions on the regularity of the data (in particular global Lipschitz continuity and boundedness of the objective are lost), state constraints, and a non-standard interaction term.
We obtain a first result on the existence of solution of the mean field game PDE system.
\end{abstract}

\smallskip \noindent


\textbf{Acknowledgements}:
All the authors have been supported by the Italian
Ministry of University and Research (MIUR), in the framework of two PRIN
projects: 2017FKHBA8 001 (The Time-Space Evolution of Economic Activities: Mathematical Models and Empirical Applications) and
20223PNJ8K (Impact of the Human Activities on the Environment and
Economic Decision Making in a Heterogeneous Setting:
Mathematical Models and Policy Implications).\\
Daria Ghilli has been supported by the INdAM-GNAMPA project “Modelli Matematici per i Processi Decisionali riguardanti la Transizione Energetica” CUP $E53C22001930001$ and by the INdAM-GNAMPA project "Modelli MFGs in Economia per lo studio della dinamica del capitale umano con spillovers spaziali" CUP $E55F22000270001.$
Giovanni Zanco has been supported by the F-NF grant ``MFGPRE: Mean-field games: regularity, path-dependency, economics''.\\
All the authors are affiliated with the GNAMPA group of INdAM (Istituto Nazionale di Alta Matematica ``Francesco Severi''). Part of this work has been carried out first while the authors were visiting the Institute of Advanced Study at Durham University, and then while Giovanni Zanco was visiting the Department of Applied Physics at Waseda University, Tokyo.

\tableofcontents

\section{Introduction}
\label{sec:intro}

The departure point of the present paper is the problem of
understanding the time-space evolution of an economic system, which is a crucial topic for economists and policy makers. The inclusion of the space dimension in economic analysis has regained relevance in the recent years. The emergence of a new economic geography is indeed one of the major events in the economic literature of the last decade (see e.g. \cite{Kru91}, \cite{Kru93}, \cite{Fujita99}).
Many papers have studied this problem under the assumption that decision about the system are taken by a unique agent (the so-called "social planner", which can be seen as the government of a nation): in this case the resulting mathematical problem is the optimal control of a Partial Differential Equation (PDE from now on), possibly stochastic (see e.g. \cite{BCZMD,BCFJET,BFFGJOEG,GozziLeocata}).
\\
However, in this context, all the agents of the economy are typically forward looking and they act following their own objectives, which do not necessarily coincide with the ones of the other agents or of the government. It is then interesting to understand if (and how) the decisions of the single agents affect and shape the time-space evolution of the whole economic system.
\\
For a large population of agents and from the methodological viewpoint, the collective behavior of such a type of agents is effectively modeled using a Mean Field Game (MFG hereafter) framework. This approach heuristically emerges as the number of agents approaches infinity, condensing the complex multi-agent interactions into a tractable system of two coupled partial differential equations, a Hamilton-Jacobi-Bellman (HJB) equation, which captures the optimal control strategy of a representative agent and a Fokker-Planck (FP) equation, which describes the evolving spatial and human capital distribution of the entire population.
The theory of MFGs has been intensively used in the last 20 years, starting from the seminal papers of Lasry-Lions \cite{LLJJM} and Huang-Caines-Malhamé \cite{HCMIEEE}, and significantly allows to shed some light on the outcome of such complex dynamics. Moreover, the MFG methodology is motivated by two principal advantages. First, solutions to the MFG system are known to approximate the Nash equilibria of the discrete, finite-agent model with high accuracy for large N. Second, it offers a crucial computational benefit. Directly solving the N-coupled optimization problems of the discrete model quickly becomes intractable, even for a modest number of agents (e.g., $N=10$). In contrast, the MFG formulation reduces the problem to a system of just the HJB and FP equations, making numerical simulation and analysis far more feasible.
This theory has been successfully employed to study the behavior of some economic models where many small homogeneous agents interact, see e.g. \cite{AchdouEtAl}. However, no paper tried to consider the case of the time-space evolution of economic variables, maybe because it displays many technical difficulties to which we return below in this introduction. We mention here the paper \cite{FRJET} where a kind of mean field approach is used, but without assuming that the small agents are forward-looking.
\\
Hence, the purpose of this paper is to provide a first mathematical and theoretical analysis of a family of MFGs arising in modeling how the strategic interactions of small agents can shape the time-space evolution of economic variables. Further research will be subsequently devoted to the economic implications of such mathematical results, and numerical analysis of the MFG under consideration in the present paper will be carried out. We mention the paper \cite{GRZet} where such numerical analysis has already been used in order to study the evolution of a spatial epidemiological model where the rate of infection depends on the distribution of the agents.
\\
As  a starting point, we assume that the small agents are homogeneous, move across space maximizing their intertemporal utility from consumption minus a disutility from displacement, and have the possibility to accumulate resources for production, for example, human capital. Hence their state variables, at any time $t\ge 0$, are the position $x(t)$ and the human capital $h(t)$. The control variables are the velocity $v(t)$ with which they move in the space, and the allocation of human capital production $s(t)$. We observe that the spatial and temporal progression of human capital is a subject of growing significance in economics. We refer e.g. to  \cite{LU88},one of the most important references for modeling human capital as an engine of growth, to  \cite{Bou08} establishing an existing literature on human capital in an epidemic setting and to \cite{Blea10} for empirical evidence. A key feature of  the model we propose is the presence of a spatial interaction terms influencing both the dynamics of the human capital and the utility of the agents. In economics such terms are called \textit{spatial spillovers}, meaning that  an individual's human capital is dynamically influenced by that of their neighbors.
\\
The resulting MFG, even in the simplest case, turns out to be very difficult and quite far from what is covered in the actual literature on the topic
(see e.g. the books \cite{CDL,CDLL}). The main reasons are the presence of the above mentioned non standard interaction terms, the weaker regularity assumptions on the data-in particular Lipschitz continuity and boundedness of the objective are lost, and, finally, a state constraint structure since human capital needs to remain positive.
In particular, the interaction term-which we will denote by $F$- has a non trivial mathematical structure, that is, is not globally Lipschitz with respect to the distribution. This feature (see Lemma \ref{barh} 3.) is stricly connected to the fact that F models the \textit{spatial spillovers} effect,  and in particular that aggregation of agents increases the human capital $h$.
Moreover, the positivity state constraints on the human capital brings the need of treating carefully the HJB and the FP equation proving suitable estimates near the border where $h=0$.
\\
We study such MFG using the PDE approach, i.e. we consider the associated PDE system of Hamilton-Jacobi-Bellman (HJB) and Fokker-Planck (FP) equations. An intensive work has been carried out on such equations, here we mention only the ones which are more related to our setting. In particular some work on Fokker-Planck equations for uncontrolled interacting systems with spatial structure and non-standard interaction term has been carried out in \cite{FPZ} and \cite{Z}.
Our main result is Theorem~\ref{thm:existence} on the existence of solutions of such PDE system. This is a departure point to understand the equilibria of the economic system, which will be the object of future research. In order to achieve this result we first study the HJB equation (see Proposition \ref{prop:MFGregularity}).  Concerning the FP equation we first analyse the properties of the drift of such equation and in Lemma \ref{lem:DpH_2} we prove that such drift is locally Lispchitz continuous for a suitable set of distributions. These results are based on  careful estimates of the underlying stochastic process modelling human capital (see e.g. point 3. in Proposition \ref{prop:MFGregularity}). After we turn to the FP equation and in order to establish well-posedness of such equation we analyse first the associated McKean-Vlasov equation. Existence of a strong solution of such equation, together with her pathwise uniqueness, are proved  in Proposition \ref{prop:MKV}. The proof is based on the estimate of the drift of the FP equation mentioned before, on a careful study of the properties of the underlying stochastic processes and on a fixed point argument. Once this result is carried out, we are ready to study the MFG in \ref{thm:existence} through again a fixed point argument. The proof is based on the results on the HJB and FP already explained.  We underline that our main existence result is local in time: in Theorem \ref{thm:existence} we require the time horizon \(T\) to belong to the admissible interval (see \eqref{eq:interval}). This  restriction was already fundamental to ensure well-posedness of the FP equation via the associated McKean-Vlasov dynamics, see Proposition \ref{prop:MKV}. This small-time condition comes from the a priori estimates in the FP and McKean-Vlasov analysis, which rely on Gronwall-type arguments and involve constants that grow exponentially in \(T\).  Hence we need to choose \(T\) sufficiently small in order to close the fixed-point/invariance estimates.

The paper is organized as follows: Section~\ref{sec:model} contains a description of the economic agent-based optimization problem, together with the assumptions we make throughout the paper, the system of partial differential equation we intend to solve and the definition of solution we use. In Section~\ref{sec:preliminary} we study separately the Hamilton-Jacobi-Bellman equation and the Fokker-Planck equation, obtaining all the necessary estimates. Eventually in Section~\ref{sec:MFG} we prove our main existence result.

 \section{The model and the assumptions}
 \label{sec:model}

 We fix the following sets and notation throughout the paper.\\
 We fix a stochastic basis $(\Omega,\sF,\left\{\sF_t\right\},\bP)$ where $\left\{\sF_t\right\}$ is a filtration of sub-$\sigma$-algebras of $\sF$ satisfying the usual conditions; we denote by $\bE$ the expectation with respect to $\bP$.\\

 We set $\bR_+=[0,+\infty)$ and $\bR_{++}=(0,+\infty)$. Our ambient space will be $\bR\times\bR_+$; we denote by $\sP_j=\sP_j(\bR\times\bR_+)$ the set of probability measures on $\bR\times\bR_+$ with finite $j$-th moment, endowed with the Wasserstein metric
\begin{equation}
  \sW_j(\mu,\nu)=\inf_{\pi\in\Gamma(\mu,\nu)}\bE\left[\left\vert (x_\mu,h_\mu)-(x_\nu,h_\nu)\right\vert^j\right]^{\frac{1}{j}},
\end{equation}
where $(x_\mu,h_\mu), (x_\nu,h_\nu)$ are random variables with joint law $\pi$ and $\Gamma(\mu,\nu)$ is the set of all couplings of $\mu$ and $\nu$.\\
By $T>0$ we denote our time horizon, which will be finite.\\
  \subsection{The economic model}
We consider an economy where there is a continuum of small interacting agents. The state variables are the agents' position and human capital. At time $t\geq 0$ they are denoted  respectively by $x(t)\in \R$ and $h(t)\in\bR_+$. The position here is chosen in $\R$ for simplicity as this paper is a first step towards the study of this type of models; our results extend easily to $\bR^d$. More realistic state spaces for the position $x$ (like $\bS^2$) could be considered in the future.\\
Each agent has control over $x(t)$ choosing, at every $t\ge 0$, the control process $v(t)$ (the velocity), and has control over $h(t)$ choosing the fraction $s(t)$ of wealth to be invested in human capital.
Moreover the agents are homogeneous and they interact with each otehr only through the distribution of their state variables.\\
The state equations of each agent on time intervals $[t_0,T]$ are as follows. The dynamics for the position is given by
\begin{equation}
\label{eq:evolxnew}
\begin{cases}
  dx(t) = v(t) dt  +\epsilon dZ(t)\quad\text{ for }t\in[t_0,T],\\
  x(t_0)=x_{0},
\end{cases}
\end{equation}
where $Z$ is a standard Brownian motion with respect to $\left\{\sF_t\right\}$, $v(\cdot)$ is the control process and $\epsilon>0$.\\
The human capital evolves according to the equation
\begin{equation}
  \label{eq:evolhnew1}
  \begin{cases}
    dh(t) = s(t) f(h(t)) F(x(t), \mu(t)) -\zeta h(t)dt +\chi h (t) dW(t)\quad\text{ for }t\in[t_0,T],\\
    h(t_0)=h_{0},
  \end{cases}
\end{equation}
Here  $W$ is a standard Brownian motion with respect to $\left\{\sF_t\right\}$, independent of $Z$, $s(\cdot)$ is the control process and $\chi>0$. Furthermore $\zeta>0$ is a constant decay factor, and, on the production function $f$ we make the following assumption:
\begin{gather*}
  f\colon\R_{+}\to\R_{+} \quad
\hbox{ is Lipschitz and increasing and such that $f(0)=0$.}
\end{gather*}
The interaction term $F$ is built as follows.
We first consider two maps $\eta_1, \eta_2\in C^2_b(\R)$ such that
$\Theta\geq\eta_i(\cdot)\geq\theta>0$ ($i=1,2$) for two strictly positive constants $\theta,\Theta$.
Then, for $x,y\in\R$ and $k\in\bR_+$ we set
\begin{equation*}
  b_1(x,y;k)=\eta_1(\vert x-y\vert)k,
  \quad b_2(x,y)=\eta_2(\vert x-y\vert),
\end{equation*}
and, for every probability measure $\mu$ on $\R\times \R_+$ with finite second moment (recall that $\sP_2$ is the set of such measures), we also set
\begin{equation*}
  \llangle \mu, b_1\rrangle(x):=\int_{\mathbb{R\times \R_+} } b_1(x,y;k)\mu(\ud y,\ud k), \qquad
  \llangle \mu, b_2\rrangle(x):=\int_{\mathbb{R}\times \R_+} b_2(x,y)\mu(\ud y,\ud k)
\end{equation*}
Given the above, we eventually define the function $F\colon\R\times\sP_2\to\bR_+$ as
\begin{equation*}
  F(x, \mu)=\frac{\llangle \mu, b_1\rrangle(x)}{\llangle \mu, b_2\rrangle(x)}.
\end{equation*}
Notice that $F$ is always well defined since the denominator is always greater than $\theta$ while we have $F(x,\mu)=0$ if and only if $\mu(\R\times\{0\})=1$.\\
From the definition of F is straightworward to notice that F models the spatial interactions and  the aggregation effect which increases the human capital of one individual. In other word,s the level of human capital of an individual is influenced bt the level of human capital of neighbourhoods. The" level" of influence is described mathematically by the functions $\eta_i$. As an example, we could consider $\eta_i$ being suitable regularizations of the indicator function of two balls, whose radius model the maximal distances of the heighbourhoods who influence the human capital of the agent considered.
\begin{remark}
  The interaction term $F$ is actually well-defined on any set of probability measures whose marginal with respect to $h$ has finite expectation; in particular it is well-defined on $\sP_1$. However since most of our arguments (including some of the properties of $F$ we exploit, see Lemma~\ref{barh}) work only in $\sP_2$, we set our model in $\sP_2$ from the beginning.
\end{remark}

To describe the optimization problem we fix a connected compact set $K\subset\bR$ such that $0\in K$ and define the set of admissible controls as
\begin{equation*}
\calk:=\left\{(v(\cdot), s(\cdot)):\Omega\times \R_+ \to K\times[0,1], \; \hbox{predictable}.\right\}
\end{equation*}
Individuals control their position in space to maximise their gain from such spillover, considering at
the same time that moving towards areas with high human capital is a costly action.
Assuming that no individual has an overwhelming influence on the systemwith respect
to others, optimisation of each individual’s position gives rise to a symmetric game
among economic agents (the individuals).

The aim of the typical agent is to maximize
\begin{multline}
  \label{eqn:J}
  J(t_0, x_{0},h_{0};v(\cdot), s(\cdot))\\:=\mathbb{E} \left [ \int_{t_0}^{T} e^{-\rho t}(u_\sigma\left( [(1-s(t)) f(h(t))]^{1-\gamma} F(x(t), \mu(t))^\gamma A( x(t))\right)-a(v(t))  ) dt  \right ]
	\end{multline}
over all admissible control process $(v,s)\in \mathcal{K}$. Here $\rho>0$ is the intertemporal discount factor while $a\colon K\to \bR$ is a given strictly convex cost function satisfying $a(0)=0$. Moreover, A represents in economics the production function or amenity of local services.

We make the following further assumptions:
\begin{gather*}
 A\colon\R\to\R
\end{gather*}
is Lipschitz and such that there exist constants $\underline{A},\overline{A}$ satisfying
\begin{equation*}
  0<\underline{A}\leq A(x)\leq \overline{A};
\end{equation*}
$u_\sigma\colon\R_{+}\to\R_{+}$ is given by
\begin{equation*}
  u_\sigma(z)=\frac{z^{1-\sigma}}{1-\sigma}
\end{equation*}
with $\sigma \in (0,1)$.\\

\begin{remark}
Heuristically, the above model with a continuum of players can be seen as the limit, when the number of agents goes to infinity, of the following game with finitely many players.
\\
For fixed $N\in\bN$, consider independent Brownian motions $(Z^i,W^i)$, $i=1,\dots,N$, and $N$ agents with respective positions $x^i(\cdot)$ and human capitals $h^i(\cdot)$ following the equations \eqref{eq:evolxnew}-\eqref{eq:evolhnew1} with $(Z^i,W^i)$ in place of $(Z,W)$. Each agent chooses controls $v^i(\cdot)$ and $s^i(\cdot)$ in $\calk$ as to maximize \eqref{eqn:J} where as $\mu(t)$ we put
\begin{equation*}
  \mu(t)=\frac1N\sum_{i=1}^N\delta_{\left(x^i(t),h^i(t)\right)},
\end{equation*}
subject to her own dynamics given by (\ref{eq:evolxnew}) and (\ref{eq:evolhnew1}).

The final aim in the study of our mean field game would be to show that its solutions (which we define below in Subsection~\ref{subsec:def}) corresponds to an approximation of a Nash equilibrium for the $N$ players game. This will be part of subsequent research.
\end{remark}

Equations (\ref{eq:evolxnew}) and (\ref{eq:evolhnew1}) can be written in vector form as follows (this is a notation we will sometimes refer to in the paper). Set
\begin{equation*}
 \mathbf{x}=(x,h),\  \mathbf{W}=(Z,W)^\top,
\end{equation*}
\begin{equation*}
  \mathbf{B}(\mathbf{x},s,v,\mu)=\mathbf{B}(x,h,s,v)=
  \begin{pmatrix}
    v\\
    s f(h) F(x,\mu)-\zeta h
  \end{pmatrix}
\end{equation*}
and
\begin{equation*}
  \mathbf{G}\colon\R\times\bR_+\to\bR\times\bR_+,\quad  \mathbf{G}(\mathbf{x})=\mathbf{G}(x,h)=
  \begin{pmatrix}
    \epsilon & 0\\
    0 & \chi h
  \end{pmatrix};
\end{equation*}
then the dynamics of $\mathbf{x}$ is given by
\begin{equation}
  \label{sys:vec}
  \begin{cases}
    \ud\mathbf{x}(t)=\mathbf{B}(\mathbf{x}(t),s(t),v(t),\mu(t))\ud t+\mathbf{G}(\mathbf{x}(t))\ud\mathbf{W}(t)\quad \text{for }t\in[t_0,T],\\
    \mathbf{x}(t_0)=(x_0,h_0).
  \end{cases}
\end{equation}

\subsection{The Mean Field Game system}
\label{subsec:def}

On $[0,T]\times\R\times\R_{+}$ we consider the forward-backward system of PDEs
\begin{equation}\label{MFGb}
\begin{cases}
  -\partial_tV(t,x,h)+\rho V(t,x,h)=H_0(p)+H_1(x,h,\mu(t),DV(t,x,h))\\
  \phantom{aaaaaaaaaaaaaaaaaaaa}+ \frac12 \epsilon^2D^2_{xx}V(t,x,h)+\frac12 \chi^2 h^2 D^2_{hh}V(t,x,h),\\
  \partial_t\mu(t)=\frac12\epsilon^2 D^2_{xx}\mu(t) + \frac12 \chi^2 D^2_{hh}\left(h^2 \mu(t)\right)\\
  \phantom{aaaaaaaaaaaaa}- D_x\left(D_pH_0(D_xV(t,x,h))\mu(t)\right)-D_h\left(D_pH_1(x,h,\mu(t),D_hV(t,x,h))\mu(t)\right),\\
\mu(0)=\mu_0, \quad V(T,x,h)=0 \quad\mbox{ in } \R\times \R_{+}.
\end{cases}
\end{equation}
whose solution is a couple $(V,\mu)$, where $V$ is a real-valued function on $[0,T]\times\R\times\R_{+}$ and $\mu$ is a function on $[0,T]$ taking values in $\sP_2$ (a precise notion of solution will be given in Definition~\ref{def:sol} below).\\

The Hamiltonians $H_0\colon\bR\to\bR$ and $H_1\colon\R\times\R_+\times\sP_2\times\R$ are given by
\begin{equation*}
  H_0(p)=\sup_{v\in K}\left\{pv-a(v)\right\}
\end{equation*}
 and
\begin{equation*}
  H_1(x,h,\mu,q)=\sup_{s\in[0,1]}\left\{\left(s f(h) F(x,\mu)-\zeta h\right)q+u_\sigma\left(A(x)\left[(1-s)f(h)\right]^{1-\gamma}F(x,\mu)^\gamma\right)\right\}.
\end{equation*}
Setting
\begin{equation*}
  \mathbf{H}(x,h,\mu,p,q)=
  \begin{pmatrix}
    H_0(p)\\ H_1(x,h,\mu,q)
  \end{pmatrix},
\end{equation*}
\begin{equation*}
  H(x,h,\mu,p,q)=\mathbf{H}(x,h,\mu,p,q)\cdot(1,1)=H_0(p)+H_1(x,h,\mu,q)
\end{equation*}
and referring to the dynamics as written in (\ref{sys:vec}), we can formulate equivalently system (\ref{MFGb}) as
\begin{equation*}
  \begin{cases}
  -\partial_tV(t,x,h)+\rho V(t,x,h)=H(x,h,\mu(t),DV(t,x,h)) + \frac12 \mathrm{Tr}\left[\mathbf{G}(x,h)\mathbf{G}^\ast(x,h)D^2V(t,x,h)\right],\\
    \partial_t\mu(t)=\frac{1}{2}\sum_{i,j} D^2_{i,j}(\mathbf{G}(x,h)\mathbf{G}^\ast(x,h)\mu(t))-\mathrm{div}\left(D_p \mathbf{H}(x,h,\mu(t), DV(t,x,h))\mu(t)\right),\\
    \mu(0)=\mu_0, \quad V(T,x,h)=0 \quad\mbox{ in } \R\times \R_{+}, \quad V(t,x,0)=0 \quad \mbox{ in } (0,T) \times \R.
  \end{cases}
\end{equation*}

We look for solutions to (\ref{MFGb}) in the following sense.
  \begin{definition}
    \label{def:sol}
    A couple $(V,\mu)$ with $V\colon[0,T]\times\bR\times\bR_+\to\bR$ and $\mu\colon[0,T]\to\sP_2$ is a solution to (\ref{MFGb}) if
\begin{enumerate}
\item it satisfies all boundary conditions;
\item $V$ is a classical $C^{1,2}((0,T)\times\bR\times\bR_{++})$ solution of the first equation in (\ref{MFGb});
\item $V$ is continuous on $[0,T]\times\bR\times\bR_{+}$;
  \item $D_xV$ and $D_hV$ are defined on $[0,T]\times\bR\times\bR_+$;
\item $\mu$ satisfies the second equation in (\ref{MFGb}) in the sense of distributions, integrated in time, i.e., for every test function $\phi\in C^{\infty}_c([0,T]\times\bR\times\bR_+)$ and every $0\leq t\leq T$ it holds
  \begin{equation}
    \label{eq:sol_weak}
    \begin{aligned}
      \int_{\bR\times\bR_+}\phi(t,x,h)&\mu(t;\ud x,\ud h) - \int_{\bR\times\bR_+}\phi(0,x,h)\mu_0(\ud x,\ud h)\\
                                      &= \int_0^t\int_{\bR\times\bR_+} \left(D_pH_0(D_xV(r,x,h))D_x\phi(r,x,h)\right.\\
                                      &\phantom{\int_0^t\int_{\bR\times\bR_+}}+\left.D_qH_1\left(x,h,\mu(r),D_hV(r,x,h)\right)D_h\phi(r,x,h)\right)\mu(r;\ud x,\ud h)\ud r\\
                                      &\phantom{\int_0^t}+ \frac12\int_0^t\int_{\bR\times\bR_+}\left(\epsilon^2 D^2_{xx}\phi(r,x,h) + \chi^2 h^2 D^2_{hh}\phi(r,x,h)\right)\mu(r;\ud x,\ud h)\ud r\\
                                      &\phantom{\int_0^t}+ \int_0^t\int_{\bR\times\bR_+}\partial_t\phi(r,x,h)\mu(r;\ud x,\ud h)\ud r.
    \end{aligned}
  \end{equation}
\end{enumerate}
\end{definition}

In the rest of the paper we make the following assumptions:
\begin{assumption}\label{ass:id}
The initial datum $\mu_0$ is in $\sP_2$.
\end{assumption}

\begin{assumption}
  \label{ass:DH_0}
  The function $a$ is such that $p\mapsto D_pH_0(p)$ is locally Lipschitz.
\end{assumption}
Assumption \ref{ass:DH_0} is satisfied, for example, if $a$ is a convex polynomial with degree at least $2$.
\begin{remark}\label{rem:H0}
Note that by Assumption \ref{ass:DH_0} and since $K$ is compact, $H_0$ is Lipschits continuous, so that $D_pH_0$ is bounded by some constant $\overline{B}$.
\end{remark}

We also introduce some shorthands: for a measure $\mu\in\sP_j$ we write

\begin{equation*}
  P_j(\mu)=\int\left\vert(y,k)\right\vert^j\mu(\ud y,\ud k),\quad M(\mu)=\int k\mu(\ud y,\ud k) \text{and}\quad M_2(\mu)=\int k^2\mu(\ud y,\ud k);
\end{equation*}
clearly $M(\mu)\leq M_2(\mu)^{\frac12}$ and, as noted before, we have $M(\mu)>0$ if and only if $\mu(\R\times\{0\})<1$ and $M(\mu)=0$ if and only if $\mu(\R\times\{0\})=1$.\\
We use similar notations when considering more measures at the same time: for a family $K$ of measures we write
\begin{equation*}
  \overline{M}(K)=\sup_{\mu\in K}\int k \mu(\ud y,\ud k)
\end{equation*}
and
\begin{equation*}
  \underline{M}(K)=\inf_{\mu\in K}\int k \mu(\ud y,\ud k).
\end{equation*}
In particular we will often deal with maps $[0,T]\ni t\mapsto \mu(t)$; in this case we write
\begin{equation*}
  \overline{M}(\mu(\cdot))=\sup_{t\in[0,T]}\int k \mu(t;\ud y,\ud k),
\end{equation*}
with analogous definition for $\underline{M}(\mu(\cdot))$. $\overline{M}_2(K)$ and $\underline{M}_2(K)$ are defined similarly, integrating the function $k^2$ in place of the function $k$.

\section{Preliminary results}
\label{sec:preliminary}
\subsection{The Hamilton-Jacobi-Bellman equation}
To study the HJB equation we fix a family of measures $\mu(\cdot)\in C^\alpha([0,T];\sP_1)$, $\alpha\in(0,1)$, and we consider the problem of choosing $(s(\cdot),v(\cdot))\in\sK$ as to maximise $J(t_0,x_0,h_0;v(\cdot),s(\cdot),\mu(\cdot))$ subject to (\ref{eq:evolxnew}) and (\ref{eq:evolhnew1}), that is, to
\begin{equation}
  \label{eq:dyn_1}
  \begin{cases}
    \ud x(t)=v(t)\ud t+\epsilon\ud W^1(t),\\
    \ud h(t)=\left[s(t)f(h(t))F(x(t),\mu(t))-\zeta h(t)\right]\ud t+\chi h(t)\ud W^2(t),\\
    x(t_0)=x_0,\ h(t_0)=h_0,
  \end{cases}
 \end{equation}
where $(x_0,h_0)\in\bR\times\bR_+$ are fixed. By standard results on stochastic differential equations, we have existence and uniqueness of a strong solution to~(\ref{eq:dyn_1}) for every fixed $(s(\cdot),v(\cdot))\in\sK$. Moreover $h(t)>0$ at all times if $h_0>0$, and $h\equiv 0$ if $h_0=0$.\\

Denote by $V$ the value function of the corresponding optimization problem, i.e.
\begin{equation}\label{eqn:V}
  V(t_0,x_0,h_0)=\sup_{(v(\cdot),s(\cdot))\in\sK}J(t_0,x_0,h_0;v(\cdot),s(\cdot),\mu(\cdot));
\end{equation}
we stress that at the present stage $\mu(\cdot)$ is fixed and $V$ depends on $\mu(\cdot)$ (although the dependence is hidden in the notation), while $\mu(\cdot)$ does not depend on $V$.

\subsubsection{Preliminary estimates on the interaction coefficient.}

	\begin{lemma}\label{barh}
		\begin{enumerate}
                \item For every fixed $\mu\in\sP_2$, $x\mapsto F(x,\mu)$ is bounded; if $M(\mu)>0$ then $x\mapsto F(x,\mu)$ is also bounded away from zero. Moreover $F(x,\mu)$ is bounded uniformly for $\mu$ in any relatively compact set of $\sP_2$.
			\item for every fixed $\mu\in\sP_2$, the function $x \mapsto F(x, \mu)$ is ($C^2$ and) Lipschitz, with Lipschitz constant depending on $\mu$. Such constant is uniform on relatively compact subsets of $\sP_2$.
			\item $\mu \mapsto F(x, \mu)$ is locally Lipschitz continuous (i.e. in any relatively compact subset of $\sP_2$), uniformly in $x$.
                        \end{enumerate}

                      \end{lemma}

\begin{proof}
  \begin{enumerate}
\item We have
\begin{equation*}
  \frac{\theta}{\Theta}M(\mu)\leq F(x,\mu)\leq\frac{\Theta}{\theta}M(\mu).
\end{equation*}
As
\begin{equation*}
  M(\mu)\leq P_1(\mu) \leq P_2(\mu)^{\frac12}
\end{equation*}
and relatively compact subsets of $\sP_2$ are $2$-uniformly integrable, the claim follows.

\item We have
\begin{eqnarray*}
  \left\vert F(x,\mu)-F(z,\mu)\right \vert& =& \left \vert\frac{\int \eta_1(|y-x|)k \mu(dy,dk)}{\int \eta_2(|y-x|)\mu(dy,dk)}-\frac{\int \eta_1(|y-z|)k \mu(dy,dk)}{\int \eta_2(|y-z|)\mu(dy,dk)}\right\vert\\  & \leq&\frac{A\int \eta_2(|y-z|)\mu(dy,dk) +B\int \eta_1(|y-z|)k \mu(dy,dk)}{\int \eta_2(|y-z|)\mu(dy,dk)\int \eta_2(|y-x|)\mu(dy,dk)},
\end{eqnarray*}
where
$$
A=\int \left \vert\eta_1(|y-x|)-\eta_1(|y-z|)\right \vert k \mu(dy,dk), \quad B=\int \left \vert\eta_2(|y-z|)- \eta_2(|y-x|)\right \vert\mu(dy,dk).
$$
Now denote by  $L_{\eta_i}$ the Lipschitz constant of $\eta_i, i=1,2$ and $L_\eta=\max \{L_{\eta_1}, L_{\eta_2}\}$.
Since
$$
A\leq L_\eta |x-z|\int k \mu(dy,dk), \quad B\leq L_\eta |x-z|
$$
we get
$$|F(x,\mu)-F(z,\mu)|\leq \frac{2L_\eta\Theta}{\theta^2}\int k\mu(\ud y,\ud k)\vert x-z\vert,
$$
thus the claim.\\

\item Let $x$ be fixed and let $\mu, \nu \in \mathcal{P}_2$. Then
  \begin{align*}
    \left\vert F(x,\mu)\right.&-\left.F(x,\nu)\right\vert=\frac{\left\vert\llangle \mu,b_1\rrangle(x)\llangle\nu,b_2\rrangle(x)-\llangle\nu,b_1\rrangle(x)\llangle\mu,b_2\rrangle(x)\right\vert}{\llangle\mu,b_2\rrangle(x)\llangle\nu,b_2\rrangle(x)}\\
                              &\leq \frac{1}{\theta^2}\Big(\llangle\mu,b_1\rrangle(x)\left\vert\llangle\mu-\nu,b_2\rrangle(x)\right\vert+\llangle\nu,b_2\rrangle(x)\left\vert\llangle\mu-\nu,b_1\rrangle(x)\right\vert\Big)
  \end{align*}
Using the Rubinstein-Kantorovich characterization of $\sW_1$ we find
\begin{align*}
  \llangle\mu,b_1\rrangle(x)\left\vert\llangle\mu-\nu,b_2\rrangle(x)\right\vert&\leq \Theta\int k\mu(\ud y,\ud k)\left\vert\int\eta_2(\vert x-y\vert) d(\mu-\nu)\right\vert\\
     &\leq \Theta M(\mu)L_{\eta_2}\sW_1(\mu,\nu)\\
     &\leq \Theta M_2(\mu)^{\frac12}L_{\eta_2}\sW_2(\mu,\nu).
\end{align*}
Choose now random variables $(y_\mu,h_\mu)$ and $(y_\nu,h_\nu)$ with law $\mu$ and $\nu$, respectively, and with joint law $\xi\in\Gamma(\mu,\nu)$, where $\xi$ is the coupling of $\mu$ and $\nu$ that attains the infimum in the definition of $\sW_2$. Then
\begin{align*}
  \llangle\nu,b_2\rrangle(x)&\left\vert\llangle\mu-\nu,
  b_1\rrangle(x)\right\vert\leq \Theta \bE\left[\left\vert\eta_1\left(\left\vert x-y_\mu\right\vert\right)h_\mu - \eta_1\left(\left\vert x-y_\nu\right\vert\right)h_\nu\right\vert\right]\\
  &\leq \Theta \left(\bE\left[\left\vert \eta_1(\left\vert x-y_\mu\right\vert)-\eta_1\left(\left\vert x-y_\nu\right\vert\right)\right\vert h_\mu\right]+\bE\left[\eta_1\left(\left\vert x-y_\mu\right\vert\right)\left\vert h_\mu-h_\nu\right\vert\right]\right)\\
  &\leq \Theta\left(L_{\eta_1}\bE\left[\left\vert y_\mu-y_\nu\right\vert^2\right]^{\frac12}\bE\left[h_\mu^2\right]^{\frac12}+\Theta\bE\left[\left\vert h_\mu-h_\nu\right\vert^2\right]^{\frac12}\right)\\
  &\leq \Theta\left(L_{\eta_1}M_2(\mu)^{\frac12}+\Theta\right)\left(\int_{\left(\bR\times\bR_+\right)^2}\left(\left\vert y-z\right\vert^2+\left\vert h-k\right\vert^2\right)\xi\left(\ud y,\ud h;\ud z,\ud k\right)\right)^{\frac12}\\
  &=\Theta\left(L_{\eta_1}M_2(\mu)^{\frac12}+\Theta\right)
  \sW_2(\mu,\nu).
\end{align*}
Thus
\begin{equation*}
\left\vert F(x,\mu)-F(x,\nu)\right\vert\leq \frac{\Theta}{\theta^2}\left(M_2(\mu)^{\frac12}
\left(L_{\eta_1}+L_{\eta_2}\right)+\Theta\right)\sW_2(\mu,\nu),
\end{equation*}
and the last assertion follows immediately from $M_2(\mu)\leq P_2(\mu)$.
\end{enumerate}
	\end{proof}

        \begin{remark}
          For fixed $x$, the function $\mu\mapsto F(x,\mu)$ is in general only locally Lipschitz continuous, not globally. Indeed $F$ is differentiable in the sense of Lions (at least for $y\neq x$) and its gradient has first component of the form $\partial_\mu F(x,\mu)(y,k)=C_1(x-y)k-C_2(x-y)\int C_3(x-y)k\mu(\ud y,\ud k)$ with $C_1,C_2,C_3$ bounded functions; the quantity $\int\partial_\mu F(x,\mu)(y,k)\mu(\ud y,\ud k)$ is unbounded over $\sP_2$, whereas it should be bounded if $F$ were globally Lipschitz (see e.g. \cite[Vol.~I, Chapter~5]{CDL}).
        \end{remark}

\subsubsection{Properties of the value function}
We will characterize the value function defined in \eqref{eqn:V} as a solution to the HJB equation in \eqref{MFGb}, that we repeat here for the reader's convenience:
\begin{equation}\label{eqn:HJB}
		\begin{cases}
		-\partial_tV+\rho V=H_0(D_xV) +H_1(t,x,h,\mu(t), D_hV)\\
		\phantom{aaaaaaaaaaaa}+ \frac12 h^2\chi^2 D^2_{hh}V+\frac12 \epsilon^2 D^2_{xx}V+
		&\mbox{ in } (0,T)\times \R\times \R_{++}\\
		V(T, x, h)=0 &\mbox{ in } \R\times \R_{+}.
		\end{cases}
		\end{equation}
         \begin{proposition}\label{prop:MFGregularity}
Let $\mu(\cdot) \in C([0,T]; \mathcal{P_1})$. Then
           \begin{enumerate}
		\item $V(t,h,0)=0$ for every $(t,x)\in[0,T]\times\R$, and $V$ is strictly positive on $[0,T]\times\R\times\R_{++}$.
			\item $V$ is increasing with respect to $h$.
			\item\label{it:visc_3}  $V$ is a continuous viscosity solution to \eqref{eqn:HJB}.
		\item
		If, in addition, $\mu(\cdot)$ is uniformly H\"older continuous on $[0,T)$, then $V \in C^{1,2}((0,T)\times \R\times \R_{++})$.
              \end{enumerate}
	\end{proposition}

We will also need some regularity of the first derivatives of $V$ up to the boundary; this is studied later in Proposition~\ref{prop:VFestimates}.

        \begin{proof}
          \begin{enumerate}

\item
	
The first statement is a direct consequence of the definition of the functional $J$ for the maximization problem, together with the assumptions on $f$, $a$ and $K$.\\
The second statemente follows simply choosing $v\equiv 0$ and any constant $s$ in $(0,1)$ and noticing that all terms are positive if $h(t)>0$.
	
\item
Set for simplicity
\begin{equation*}
  U_\sigma(x,h,s,\mu)=u_\sigma\left(\left[(1-s)f(h)\right]^{1-\gamma}F(x,\mu)^\gamma A(x)\right).
\end{equation*}
	Lat $h_1>h_2>0$ and $x_0, t_0 \in \bR\times \R_+$. For all $\eps>0$ there exists $(v_\eps(\cdot), s_\eps(\cdot)) \in \mathcal{K}$ such that
\begin{equation*}
      	V(t_0, x_0, h_2)\leq \mathbb{E}\left[\int_{t_0}^{T}e^{-\rho t}\left(U_\sigma(x_\eps(t),h_\eps(t),s_\eps(t),\mu(t))-a(v_\eps(t))\right)\ud t\right]+\eps,
      \end{equation*}
      where $x_\eps(\cdot), h_{\eps,2}(\cdot)$ are the trajectories with initial data $x_0, h_2$ respectively, controlled by $(v_\eps(\cdot), s_\eps(\cdot))$.	If $h_{\eps,1}(t)$ is the trajectory with initial condition $h_1$ controlled by $(v_\eps(\cdot), s_\eps(\cdot))$, we have
	\begin{multline}\label{eq:vincr}
	V(t_0, x_0, h_1)-V(t_0, x_0, h_2)\\\geq \mathbb{E}\left[\int_{t_0}^{T}e^{-\rho t}\left(U_\sigma\left(x_\eps(t),h_{\eps,1}(t), s_\eps(t),\mu(t)\right)-U_\sigma\left(x_\eps(t),h_{\eps,2}(t), s_\eps(t),\mu(t)\right)\right)\ud t\right]-\eps.
      \end{multline}
Since $h\mapsto U_\sigma(x,h,s,\mu)$ is increasing and, by standard results on SDEs with Lipschitz coefficients, $\P(h_{\eps,1}(t)\geq h_{\eps,2}(t)\forall t\in[0,T])=1$, we get the claim by arbitrariness of $\eps$.

\item First observe that, by easy adaptations of standard estimates for SDEs based on the Burkholder-Davis-Gundy inequality and Gronwall's lemma, we have that for any admissible control $(v(\cdot),s(\cdot))$ and any random initial condition $(x_0,h_0)$ in $L^p$, $p\geq 2$, the corresponding solution $(x(\cdot),h(\cdot))$ to (\ref{eq:dyn_1}) satisfies for every $t\in[t_0,T]$
  \begin{equation}
    \label{eq:est_h_1}
    \begin{aligned}
      \bE\left[\sup_{s\in[t_0,t]} \vert x(s)\vert^p\right]&\leq 3^{p-1}\bE\left[\left\vert x_0\right\vert\right]+3^{p-1}(\max K)^p (t-t_0)^p+\epsilon^p B_{0,p}(t-t_0)^{\frac{p}{2}},\\
      \bE\left[\sup_{s\in[t_0,t]}h(s)^p\right]&\leq 4^{p-1} e^{C_{2,p} (t-t_0)}\bE \left[h_0^p\right],
    \end{aligned}
  \end{equation}
  where
  \begin{gather}
    \label{eq:constants_B_0}
    B_{0,p}= \left(\frac{p^3}{2p-2}\right)^{\frac{p}{2}},\\
        \label{eq:constants_C_2}
        C_{2,p}=C_{2,p}(T,\mu(\cdot),p)=C_{1,p}L_f^p+(4T)^{p-1}\zeta^p+4^{p-1}\left(\frac{p^3}{2p-2}\right)^{\frac{p}{2}}\chi^p T^{\frac{p-2}{2}},\\
            \label{eq:constants_C_1}
    C_{1,p}=C_{1,p}(T,\mu(\cdot),p)=(4T)^{p-1}\left(\frac{\Theta}{\theta}\overline{M}(\mu(\cdot))\right)^p.
  \end{gather}
  Moreover if $(x_i(\cdot),h_i(\cdot))$, $i=1,2$ solve (\ref{eq:dyn_1}) with random initial conditions $(x_i,h_i)$ respectively and with the same control process $(v(\cdot),s(\cdot))$, we have for every $t\in[t_0,T]$ and $p\geq 2$
  \begin{equation}
    \label{eq:est_h_3}
    \bE\left[\sup_{s\in[t_0,t]}\left\vert x_1(s)-x_2(s)\right\vert^p\right]= \bE\left[\left\vert x_1-x_2\right\vert^p\right]
  \end{equation}
  and
  \begin{multline}
    \label{eq:est_h_2}
  \bE\left[\sup_{s\in[t_0,t]}\left\vert h_1(s)-h_2(s)\right\vert^2\right]\\\leq 4e^{2 C_{2,2} (t-t_0)}\left(\bE\left[\left\vert h_1-h_2\right\vert^2\right]+64 T C_{1,2}L_f^2L_\eta^2e^{\frac12 C_{2,4} (t-t_0)}\bE\left[h_2^4\right]^{\frac12}\bE\left[\left\vert x_1-x_2\right\vert^4\right]^{\frac12}\right).
  \end{multline}

  Fix now $t_0$ and deterministic initial conditions $x_1,h_1$; then there exist controls $(v_\eps(\cdot),s_\eps(\cdot))$ such that for the corresponding controlled processes $(x_{\eps,1}(\cdot),h_{\eps,1}(\cdot))$ we have
  \begin{equation*}
    V(t_0,x_1,h_1)\leq\bE\left[\int_{t_0}^Te^{-\rho t}\left(U_\sigma\left(x_{\eps,1}(t),h_{\eps,1}(t),s_\eps(t),\mu(t)\right)-a(v_\eps(t))\right)\ud t\right]+\eps.
    \end{equation*}
    Consider the processes $(x_{\eps,2}(\cdot),h_{\eps,2}(\cdot))$ that have a couple $(x_2,h_2)$ as initial conditions and are controlled by $(v_\eps(\cdot),s_\eps(\cdot))$. We then have
    \begin{align}
      \label{eq:V_cont_est_1}      V&(t_0,x_1,h_1)-V(t_0,x_2,h_2)\\
\nonumber                                                 &\leq\bE\left[\int_{t_0}^Te^{-\rho t}\left\vert U_\sigma\left(x_{\eps,1}(t),h_{\eps,1}(t),s_\eps(t),\mu(t)\right)-U_\sigma\left(x_{\eps,2}(t),h_{\eps,2}(t),s_\eps(t),\mu(t)\right)\right\vert\ud t\right]+\eps\\
\nonumber                                   &\leq \bE\left[\int_{t_0}^Te^{-\rho t}\left\vert U_\sigma\left(x_{\eps,1}(t),h_{\eps,1}(t),s_\eps(t),\mu(t)\right)-U_\sigma\left(x_{\eps,2}(t),h_{\eps,1}(t),s_\eps(t),\mu(t)\right)\right\vert\ud t\right]\\
\nonumber                                   &\phantom{\leq}+\bE\left[\int_{t_0}^Te^{-\rho t}\left\vert U_\sigma\left(x_{\eps,2}(t),h_{\eps,1}(t),s_\eps(t),\mu(t)\right)-U_\sigma\left(x_{\eps,2}(t),h_{\eps,2}(t),s_\eps(t),\mu(t)\right)\right\vert\ud t\right]+\eps
    \end{align}
    Set $\eta=(1-\gamma)(1-\sigma)$; we have, almost surely,
\begin{multline*}
  \left\vert U_\sigma\left(x_{\eps,1}(t),h_{\eps,1}(t),s_\eps(t),\mu(t)\right)-U_\sigma\left(x_{\eps,2}(t),h_{\eps,1}(t),s_\eps(t),\mu(t)\right)\right\vert\\\leq  \frac{f\left(h_{\eps,1}(t)\right)^\eta}{1-\sigma} \left(A(x_{\eps,1}(t))^{1-\sigma}\left\vert F(x_{\eps,1}(t),\mu(t))^{\gamma(1-\sigma)}-F(x_{\eps,2}(t),\mu(t))^{\gamma(1-\sigma)}\right\vert\right.\\\left.+F(x_{\eps,2}(t),\mu(t))^{\gamma(1-\sigma)}\left\vert A(x_{\eps,1}(t))^{1-\sigma}-A(x_{\eps,2}(t))^{1-\sigma}\right\vert\right).
\end{multline*}
As $0<\gamma(1-\sigma)\leq 1$ we have, using the fact that the equations for $x_{\eps,j}(t)$ can be solved pathwise,
\begin{align*}
  \left\vert F(x_{\eps,1}(t),\mu(t))^{\gamma(1-\sigma)}-F(x_{\eps,2}(t),\mu(t))^{\gamma(1-\sigma)}\right\vert\leq \left(\frac{2L_\eta\Theta}{\theta^2}\overline{M}(\mu(\cdot))\right)^{\gamma(1-\sigma)}\left\vert x_1-x_2\right\vert^{\gamma(1-\sigma)}.
\end{align*}
Therefore, since $A$ is Lispchitz with Lipschitz constant $L_A$, we have almost surely
\begin{multline*}
  \left\vert U_\sigma\left(x_{\eps,1}(t),h_{\eps,1}(t),s_\eps(t),\mu(t)\right)-U_\sigma\left(x_{\eps,2}(t),h_{\eps,1}(t),s_\eps(t),\mu(t)\right)\right\vert\\\leq\frac{f\left(h_{\eps,1}\right)^\eta}{1-\sigma}\left(\frac{\Theta}{\theta}\overline{M}(\mu(\cdot))\right)^{\gamma(1-\sigma)}\left(\overline{A}^{1-\sigma}\left(2\frac{L_\eta}{\theta}\right)^{\gamma(1-\sigma)}\left\vert x_1-x_2\right\vert^{\gamma(1-\sigma)}+L_A^{1-\sigma}\left\vert x_1-x_2\right\vert^{1-\sigma}\right),
\end{multline*}
yielding (by (\ref{eq:est_h_1}) and Jensen's inequality)
\begin{multline}
  \label{eq:est_delta_V_1}
\bE\left[\int_{t_0}^T e^{-\rho t}\left\vert U_\sigma\left(x_{\eps,1}(t),h_{\eps,1}(t),s_\eps(t),\mu(t)\right)-U_\sigma\left(x_{\eps,2}(t),h_{\eps,1}(t),s_\eps(t),\mu(t)\right)\right\vert\right]\ud t\\\leq \frac{2^\eta}{1-\sigma} L_f^\eta e^{\frac{\eta}{2}C_{2,2} (T-t_0)}(T-t_0) h_1^\eta\left(\frac{\Theta}{\theta}\overline{M}(\mu(\cdot))\right)^{\gamma(1-\sigma)}\\ \cdot\left(\overline{A}^{1-\sigma}\left(2\frac{L_\eta}{\theta}\right)^{\gamma(1-\sigma)}\left\vert x_1-x_2\right\vert^{\gamma(1-\sigma)}+L_A^{1-\sigma}\left\vert x_1-x_2\right\vert^{1-\sigma}\right).
\end{multline}
By (\ref{eq:est_h_2}) we also have
\begin{multline}
    \label{eq:est_delta_V_2}
  \bE\left[\int_{t_0}^T e^{-\rho t}\left\vert U_\sigma\left(x_{\eps,1}(t),h_{\eps,1}(t),s_\eps(t),\mu(t)\right)-U_\sigma\left(x_{\eps,1}(t),h_{\eps,2}(t),s_\eps(t),\mu(t)\right)\right\vert\ud t\right]\\\leq \frac{1}{1-\sigma}\overline{A}^{1-\sigma}\left(\frac{\Theta}{\theta}\overline{M}(\mu(\cdot))\right)^{\gamma(1-\sigma)}\left(2L_f\right)^\eta e^{\eta C_{2,2} (T-t_0)}(T-t_0)\\ \cdot\left(\left\vert h_1-h_2\right\vert^2+32TC_{1,2}L_f^2L_\eta^2h_2^2 e^{C_2 (T-t_0)}\left\vert x_1-x_2\right\vert^2\right)^{\frac{\eta}{2}}.
  \end{multline}
  As the reverse inequality in (\ref{eq:V_cont_est_1}) holds analogously and $\eps$ is arbitrary, continuity of $V$ with respect to $(x,h)$ follows.\\
	
  We proceed to prove continuity of $V$ with respect to $t$, locally uniformly with respect to $(x,h)$.\\
Again by standard SDEs estimates we have that, for any admissible control $(v(\cdot),s(\cdot))$, every couple of times $t_0\leq t_1\leq t_2\leq T$ and any $p\geq 2$, the corresponding solution to (\ref{eq:dyn_1}) satisfy
\begin{equation}
  \label{eq:est_x_t}
  \bE\left[\left\vert x(t_2)-x(t_1)\right\vert^p\right]\leq 2^{p-1} (K^pT^{\frac{p}{2}}+\epsilon^p)(t_2-t_1)^{\frac{p}{2}}
\end{equation}
and
\begin{equation}
  \label{eq:est_h_t}
  \bE\left[\left\vert h(t_2)-h(t_1)\right\vert^2\right]\leq 12 C_{2,2} e^{C_{2,2}(t_2-t_1)}h_0^2(t_2-t_1).
\end{equation}
  Fix $x_0 \in \bR, h_0 \in \R_+$ and take $t_1, t_2 \in \R_+, 0\leq t_1<t_2<T$. By the Dynamic Programming Principle we have that for $\eps>0$ there exist controls $(v_\eps(\cdot), s_\eps(\cdot))\in \calk$ with the corresponding solutions $(x_\eps(\cdot), h_\eps(\cdot))$  such that $x_\eps(t_1)=x_0, h_\eps(t_1)=h_0$ for which
\begin{align*}
  V(t_1, x_0, h_0)- V(t_2, x_\eps(t_2),&h_\eps(t_2))\\
  &\leq \bE\left[\int_{t_1}^{t_2}e^{-\rho t}U_\sigma(x_\eps(t), h_\eps(t), v_\eps(t), s_\eps(t), \mu(t))\, \ud t\right] +\eps,\\
  &\leq\frac{\overline{A}^{1-\sigma}}{1-\sigma}\left(\frac{\Theta}{\theta}\overline{M}(\mu(\cdot))\right)^{\gamma(1-\sigma)}\left(2L_f\right)^\eta h_0^\eta e^{\frac{\eta}{2} C_2(t_2-t_1)}(t_2-t_1)+\eps.
              \end{align*}		

Now set denote $\tilde x=x_\eps(t_2), \tilde h=h_\eps(t_2)$ and write
  \begin{equation}
    \label{eq:V_cont_est_2}
    V(t_1,x_0,h_0)-V(t_2,x_0,h_0)= \left[V(t_1,x_0,h_0)-V(t_2,\tilde{x},\tilde{h})\right]+\left[V(t_2,\tilde{x},\tilde{h})-V(t_2,x_0,h_0)\right].
  \end{equation}
The term in the first bracket of the right hand side has just been estimated. To study the term in the second bracket fix $\delta >0$ and take controls $v_\delta(\cdot), s_\delta (\cdot)$ with the corresponding trajectories $x_\delta(\cdot),h_\delta(\cdot)$ starting from $\tilde{x},\tilde{h}$ at time $t_2$ and such that
\begin{equation*}
V(t_2, \tilde x, \tilde h)\leq \bE\int_{t_2}^{T} e^{-\rho t}U_\sigma\left(x_\delta(t), h_\delta(t), v_\delta(t), s_\delta(t), \mu(t)\right)\ud t +\delta.
\end{equation*}
If $x_{\delta,0}(\cdot), h_{\delta,0}(\cdot)$ are the trajectories controlled by $v_\delta(\cdot),s_\delta(\cdot)$ and starting from $x_0,h_0$ at time $t_2$, we have, arguing as in (\ref{eq:V_cont_est_1}), (\ref{eq:est_delta_V_1}), (\ref{eq:est_delta_V_2}) and using first (\ref{eq:est_h_1}), (\ref{eq:est_h_3}), (\ref{eq:est_h_2}) and eventually (\ref{eq:est_x_t}), (\ref{eq:est_h_t}),
\begin{align*}
  V&(t_2,\tilde{x},\tilde{h})-V(t_2,x_0,h_0)\\
                            &\leq \delta + \frac{2^\eta}{1-\sigma}e^{\frac{\eta}{2}C_{2,2}T}\bE\left[\tilde{h}^2\right]^{\frac{\eta}{2}} L_f^\eta T \left(\frac{\Theta}{\theta}\overline{M}(\mu)\right)^{\gamma(1-\sigma)}\\
                            &\phantom{aa}\cdot\left(\overline{A}^{1-\sigma}\left(2\frac{L_\eta}{\theta}\right)^{\gamma(1-\sigma)}\bE\left[\left\vert \tilde{x}-x_0\right\vert^2\right]^{\frac{\gamma(1-\sigma)}{2}}+L_A^{1-\sigma}\bE\left[\left\vert \tilde{x}-x_0\right\vert^2\right]^{\frac{1-\sigma}{2}}\right)\\
                            &\phantom{aa}+\frac{\overline{A}^{1-\sigma}}{1-\sigma}\left(\frac{\Theta}{\theta}\overline{M}(\mu)\right)^{\gamma(1-\sigma)}\left(2L_f\right)^\eta e^{\eta C_{2,2} T}T\\
                            &\phantom{aa}\cdot\left(\bE\left[\left\vert \tilde{h}-h_0\right\vert^2\right]+64TC_{1,2}L_f^2L_\eta^2 e^{\frac12 C_{2,4} T}h_0^2\bE\left[\left\vert \tilde{x}-x_0\right\vert^4\right]^{\frac12}\right)^{\frac{\eta}{2}}\\
                            &\leq \delta + \frac{2^{1+\eta}}{1-\sigma}e^{\eta C_{2,2}T}h_0^\eta L_f^\eta T \left(\frac{\Theta}{\theta}\overline{M}(\mu)\right)^{\gamma(1-\sigma)}\\
                            &\phantom{aa}\cdot\left(\overline{A}^{1-\sigma}\left(\sqrt{8}\sqrt{K^2T+\epsilon^2}\frac{L_\eta}{\theta}\right)^{\gamma(1-\sigma)}(t_2-t_1)^{\frac{\gamma(1-\sigma)}{2}}+\left(\sqrt{2}\sqrt{K^2T+\epsilon^2}L_A\right)^{1-\sigma}(t_2-t_1)^{\frac{1-\sigma}{2}}\right)\\
                            &\phantom{aa}+\frac{\overline{A}^{1-\sigma}}{1-\sigma}\left(\frac{\Theta}{\theta}\overline{M}(\mu)\right)^{\gamma(1-\sigma)}\left(2L_f\right)^\eta e^{\eta C_{2,2} T}T\\
                            &\phantom{aaa}\cdot\left(12 C_{2,2}e^{C_{2,2}T}h_0^2(t_2-t_1)+64TC_{1,2}L_f^2L_\eta^2 e^{\frac12 C_{2,4} T}h_0^2\sqrt{8}\sqrt{T^2K^4+\epsilon^4}(t_2-t_1)\right)^{\frac{\eta}{2}}.
\end{align*}
Eventually from (\ref{eq:V_cont_est_2}) and the previoius computations we obtain that
\begin{multline*}
  V(t_2,x_0,h_0)-V(t_1,x_0,h_0)\\\leq \delta + \eps + \widetilde{C}_1 (t_2-t_1) + \widetilde{C}_2 (t_2-t_1)^{\frac{\gamma(1-\sigma)}{2}} + \widetilde{C}_3(t_2-t_1)^{\frac{1-\sigma}{2}} + \widetilde{C}_4(t_2-t_1)^{\frac{\eta}{2}};
\end{multline*}
as the constants appearing in this last formula can be bounded uniformly for $(x_0,h_0)$ in bounded sets, continuity of $V$ with respect to $t$, locally uniformly with respect to $(x.h)$, follows from arbitrariness of $\delta$ and $\eps$. The joint continuity in $(t,x,h)$ is then a consequence of Dini's Theorem.\\
Once continuity is established, it is standard to prove that the value function is a viscosity solution of the HJB equation in \eqref{MFGb}, see for example \cite{Users}.
\item 		The argument is  classical using regularity results for uniformly parabolic equations with uniformly H\"older coefficients. We give the proof for completeness. Note that we prove the above mentioned regularity of $V$ when $h$ belongs to the space $\R_{++}$, since the HJB equation in \eqref{MFGb} degenerates at $h=0$.\\
 Let $t_0 \in \R_{++}, x_0 \in \bR, h_0 \in\R_{++}$ and take first $\eps>0$ such that $ h_0-\eps\in \R_{++}$ and $(t_0-\eps, t_0+\eps)\in [0,T)$. Define
\begin{equation*}
  \mathcal{D}_\eps(t_0, x_0, h_0)=(t_0-\eps, t_0+\eps)\times (x_0-\eps, x_0+\eps) \times (h_0-\eps, h_0+\eps)
\end{equation*}
and denote by $\partial\mathcal{D}_\eps(t_0,x_0,h_0)$ its boundary. By the assumption on $\mu(t)$, the HJB equation in \eqref{MFGb} is a uniformly parabolic equation in $\mathcal{D}_\eps(t_0,x_0,h_0)$ with uniformly H\"older coefficients. Then we have uniqueness of viscosity solutions by the results in \cite{Users} and by Theorem 12.22 of \cite{Lieb} (with the assumptions of Theorem 12.16 in the same reference), existence of a solution in the class $C^{1,2}(\mathcal{D}_\eps(t_0, x_0, h_0))$. This classical solution is also a viscosity solution so that, by the uniqueness of viscosity solutions, it must coincide with $V$. Therefore we conclude that $V\in C^{1,2}(\mathcal{D}_\eps(t_0, x_0, h_0))$ and hence by the arbitrariness of $t_0, x_0, h_0$, we have that $V \in C^{1,2}((0,T)\times \R \times \R_{++})$.\\
\end{enumerate}		
\end{proof}

\begin{proposition}\label{prop:VFestimates}
		Let $V: [0,T]\times \mathbb{R}\times \R_+$ be the  value function V defined in \eqref{eqn:V}. Then for any $\mu\in C([0, T];  \mathcal{P}_1)$
we have that, for all $t \in [0,T]$,
$V(t,\cdot)$ belongs to $C^1(\mathbb{R}\times \R_+)$
and
\begin{equation}\label{eq:ineqhDV}
\sup_{(t,x,h)\in [0,T]\times\mathbb{R}\times \R_+}\left\{
|D_x V(t,x,h)|+|h D_h V(t,x,h)|
\right\}<+\infty
\end{equation}
	\end{proposition}
	\begin{proof}
Consider, for given $\phi\in C^0(\mathbb{R}\times \R_{+}$), the HJB equation (i.e. the first equation of
\eqref{eqn:HJB} with generic final datum $\phi$):
\begin{equation}\label{HJBperreg}
\begin{cases}
 -\partial_tV+\rho V=H_1(t,x,h,\mu(t), D_hV)+H_0(D_xV)&\\
 \phantom{aaaaaaaaaaaa}+ \frac12 \chi^2h^2 D^2_{hh}V+\frac12 \epsilon^2 D^2_{xx}V &\mbox{ in } (0,T)\times \mathbb{R}\times \R_{++}\\
 V(T,x,h)=\phi(x,h) &\mbox{ in } \mathbb{R}\times \R_{+}\\
 V(t,x,0)=0 &\mbox{ in } (0,T) \times \mathbb{R}
\end{cases}
\end{equation}
By reversing time, i.e. calling $v(t,x,h):=V(T-t,x,h)$ this PDE becomes
\begin{equation}\label{HJBperregtimerev}
\begin{cases}
\partial_tv=\rho v +H_0(D_xv) + H_1(T-t,x,h,\mu(T-t), D_hv)&\\
\phantom{aaaaaaaaa}+\frac12 \epsilon^2 D^2_{xx}v
+\frac12 \chi^2h^2 D^2_{hh}v&\mbox{ in } (0,T)\times \mathbb{R}\times \R_{++}\\
 v(0,x,h)=\phi(x,h) &\mbox{ in } \mathbb{R}\times \R_{+}\\
 v(t,x,0)=0 &\mbox{ in } (0,T) \times \mathbb{R}
\end{cases}
\end{equation}
Now we take the following exponential change of variable
$e^y=h$ or $y=\ln h$. We write
$w(t,x,y)=v(t,x,e^y)$.
With this change of variable the above HJB PDE,
after simple computations,
becomes
\begin{equation}\label{HJBperregCV}
\begin{cases}
\partial_tw= \rho w
  +H_0(D_x w)+ H_1(T-t,x,e^y,\mu(T-t), e^{-y}D_y w)& \\
  \phantom{aaaaaaaaaa}+\frac12 \epsilon^2 D^2_{xx}w
+\frac12 \chi^2 D^2_{yy}w -\frac{1}{2}\chi D_{y}w  &\mbox{ in } (0,T)\times \mathbb{R}\times \R\\
 w(0,x,y)=\phi(x,e^y) &\mbox{ in } \mathbb{R}\times \R\\
w(t,x,-\infty)=0 & \mbox{ in } (0,T) \times \mathbb{R}
\end{cases}
\end{equation}
Now, we consider the linear part of such a PDE and
we call, formally, $\mathcal{L}$ the corresponding linear differential operator, writing,
for any regular enough map
$\psi:\mathbb{R}\times \R\to \R$,
$$
\mathcal{L}\psi:=\frac12 \epsilon^2 D^2_{xx}\psi
+\frac12 \chi^2 D^2_{yy}\psi -\frac{1}{2}\chi D_{y}w +\rho \psi
$$
This is a nondegenerate elliptic operator which can be studied using the theory of analytic semigroups
well developed e.g. in the first three chapters of the book of Lunardi \cite{Lunardi}.
In this way we can see the PDE \eqref{HJBperregCV} as a semilinear equation in a Banach space $\mathcal{X}$ (driven by the operator $\mathcal{L}$) of the type studied in Chapter 7 of the above book of Lunardi \cite{Lunardi}.
To be more precise here the space $\mathcal{X}$ is the weighted space
$$
C_w(\R^2,\R):=\left\{
f\in C(\R^2,\R)\,: \; \hbox{ the map $z\mapsto e^{(1-\gamma)(1-\sigma)z}f(z)$ belongs to $C_b(\R^2,\R)$}
\right\}
$$
where the weight is chosen in way that the value function, after the change of variable, belongs to that space.
We can then apply the results of
\cite[Chapter 7]{Lunardi}, more precisely
Theorem 7.1.3 and Proposition 7.2.1.
Such results establish that, under suitable assumptions on
the nonlinearities the above PDE \eqref{HJBperregCV} has a unique mild solution satisfying the estimate
$$
\sup_{t\in [0,T]}\|w(t,\cdot)\|_{C^1_w(\mathbb{R}\times \R)}<+\infty
$$
where
$$
C^1_w(\R^2,\R):=\left\{
f\in C^1(\R^2,\R)\,: \; \hbox{ the map $z\mapsto e^{(1-\gamma)(1-\sigma)z}f(z)$ belongs to $C^1_b(\R^2,\R)$}
\right\}.
$$
This implies, in particular (reversing the change of variable) that $V$ satisfies \eqref{eq:ineqhDV}, as required.

Now the value function $V$ is a classical solution of the HJB equation \eqref{HJBperreg}, hence, after the above changes of variables in $t$ and in $y$ it must also be a classical solution of \eqref{HJBperregCV}.
This implies that it is also a mild solution of
\eqref{HJBperregCV}, hence that it satisfies the required inequality.

It then remains to check that the nonlinearity here satisfies the assumptions of Theorem 7.1.3 and Proposition 7.2.1 of
\cite[Chapter 7]{Lunardi}.
This is immediate consequence of the Remark 2.4 (for the part of $H_0$) and of the subsequent Lemma 3.4 (for the part of $H_1$).
\end{proof}

\subsubsection{Properties of the Hamiltonian.}
The function $H_1$ and its derivatives can be characterized explicitly, leading to some useful properties.
\begin{lemma}\label{lem:DpH}
  There exist continuous functions $g,g_1\colon\bR_+\to\bR_+$ such that, for every $x \in \R, h \in \R_+, \mu \in \mathcal{P}_1$ and $p\in\R_+$,
  \begin{equation*}
    \left\vert H_1(x,h,\mu,p)\right\vert\leq g(M(\mu))(hp+h^\eta)
  \end{equation*}
  and
  \begin{equation*}
    \left\vert D_pH_1(x,h,\mu,p)\right\vert\leq g_1(M(\mu))h.
  \end{equation*}
\end{lemma}
\begin{proof}
  Let us set, for every fixed tuple $(x,h,\mu,p)$
  \begin{equation*}
    k(s)=H_1^{CV}(s,x,h,\mu,p)=sf(h)F(x,\mu)p-\zeta hp+u_\sigma\left(A(x)(1-s)^{1-\gamma}f(h)^{1-\gamma}F(x,\mu)^\gamma\right);
  \end{equation*}
  we thus have
  \begin{equation*}
    H_1(x,h,\mu,p)=\sup_{s\in[0,1]}k(s).
  \end{equation*}
  Clearly $H_1(x,0,\mu,p)=0$. If $M(\mu)=0$ then $\mu(\R\times\{0\})=1$ so that $H_1(x,h,\mu,p)=-\zeta hp$; in these cases the claim is evidently true, and morevoer $D_pH_1$ does not depend on $p$. Therefore we now assume that $h>0$ and $M(\mu)>0$.

  As we are interested in the dependence on $p$, to keep notation simple we set
  \begin{equation*}
    a=a(x,h,\mu):=f(h)F(x,\mu),\quad b=b(x,h,\mu):=(1-\sigma)u_\sigma\left(A(x)f(h)^{1-\gamma}F(x,\mu)^{\gamma}\right),
  \end{equation*}
  so that we can write
  \begin{equation*}
    k(s)=aps-\zeta hp+\frac{(1-s)^\eta}{1-\sigma}b\ .
  \end{equation*}
  Then
  \begin{equation*}
  k^\prime(s)=ap-(1-\gamma)b(1-s)^{\eta-1},
\end{equation*}
which is zero for the only value
\begin{equation*}
  \bar{s}(p)=1-\left(\frac{(1-\gamma)b}{ap}\right)^{\frac{1}{1-\eta}},
\end{equation*}
which is a point of local maximum for $k$.\\
We have $\bar{s}(p)\in[0,1]$ if and only if
\begin{equation*}
  0\leq (1-\gamma)\frac{b}{ap}=(1-\gamma)A(x)^{1-\sigma}f(h)^{\eta-1} F(x,\mu)^{\gamma(1-\sigma)-1}\frac1p\leq 1;
\end{equation*}
we set
\begin{equation*}
  p_0=(1-\gamma)\frac{b}{a}=(1-\gamma)A(x)^{1-\sigma}f(h)^{\eta-1} F(x,\mu)^{\gamma(1-\sigma)-1},
\end{equation*}
so that $p_0>0$ for every $x,h,\mu$ and $\bar{s}(p)\in[0,1]$ if and only if $p\geq p_0$. Clearly for $p\leq p_0$ we have that $\sup_{s\in[0,1]}k(s)=\frac{b}{1-\sigma}-\zeta hp$ is attained at $s=0$. Therefore we get
\begin{equation*}
  \bar{s}(p)=
  \begin{cases}
    0&\text{ for }p\leq p_0,\\
    1-\frac{p_0}{p}^{\frac{1}{1-\eta}}&\text{ for }p>p_0,
  \end{cases}
\end{equation*}
so that
\begin{align}
\nonumber  H_1(&x,h,\mu,p)=
  \begin{cases}
    \frac{b}{1-\sigma}-\zeta hp&\text{ for }p\leq p_0\\
    \bar{s}(p)a-\zeta hp+\frac{b}{1-\sigma}\left(1-\bar{s}(p)\right)^\eta&\text{ for }p> p_0
  \end{cases}\\
\label{eq:al_H1_1}                &=
                  \begin{cases}
                    u_\sigma\left(A(x)f(h)^{1-\gamma}F(x,\mu)^\gamma\right)-\zeta hp&\text{ for }p\leq p_0\\
                    \left(1-\left(\frac{p_0}{p}\right)^{\frac{1}{1-\eta}}\right)f(h)F(x,\mu)p-\zeta hp+u_\sigma\left(A(x)f(h)^{1-\gamma}F(x,\mu)^\gamma\right)\left(\frac{p_0}{p}\right)^{\frac{\eta}{1-\eta}}&\text{ for }p> p_0
                  \end{cases}\\
  \label{eq:al_H1_2}                &=
                                      \begin{cases}
                                            \frac{1}{1-\sigma}A(x)^{1-\sigma}f(h)^{\eta}F(x,\mu)^{\gamma(1-\sigma)}-\zeta hp&\text{ for }p\leq p_0\\
    \left(f(h)F(x,\mu)-\zeta h\right)p+\frac{1-\eta}{\eta}p_0^{\frac{1}{1-\eta}}f(h)F(x,\mu)p^{-\frac{\eta}{1-\eta}}&\text{ for }p> p_0.
                                      \end{cases}
\end{align}
  From (\ref{eq:al_H1_1}) it is immediately seen that $H_1(x,h,\mu,\cdot)$ is continuous and bounded in $p$ (and actually bounded uniformly with respect to $(h,\mu)$ in bounded sets), because $p_0>0$. Morevoer
  \begin{equation}
    \label{eq:al_DH1_1}
    D_pH_1(x,h,\mu,p)=  -\zeta h+ f(h)F(x,\mu)\left(1-\left(\frac{p_0}{p}\right)^{\frac{1}{1-\eta}}\right)\ind_{(p_0,+\infty)}(p),
  \end{equation}
  which is also continuous and bounded in $p$ for fixed $(x,h,\mu)$ (and, again, uniformly bounded for $(h,\mu)$ in bounded sets).\\
  We have from (\ref{eq:al_H1_2}) that for $p\leq p_0$
\begin{equation*}
  \left\vert H_1(x,h,\mu,p)\right\vert\leq \frac{\overline{A}^{1-\sigma}}{1-\sigma}\left(\frac{\Theta}{\theta}M(\mu)\right)^{\gamma(1-\sigma)}h^\eta+\zeta hp.
\end{equation*}
When $p>p_0$, since $p^{-\frac{\eta}{1-\eta}}$ is decreasing, we have
  \begin{align*}
    \left\vert H_1(x,h,\mu,p)\right\vert&\leq \left(\frac{\Theta}{\theta}M(\mu)-\zeta\right)hp
    +\frac{1-\eta}{\eta}p_0f(h)F(x,\mu)\\
    &=\left(\frac{\Theta}{\theta}M(\mu)-\zeta\right)
    hp+(1-\gamma)\frac{1-\eta}{\eta}\overline{A}^{1-\sigma}
    \left(\frac{\Theta}{\theta}M(\mu)\right)^{\gamma(1-\sigma)}h^\eta.
  \end{align*}
  Therefore
  \begin{equation*}
    g(z)=\max\left\{\zeta,\overline{A}^{1-\sigma}\left(\frac{\Theta}{\theta}z\right)^{\gamma(1-\sigma)},(1-\gamma)\frac{1-\eta}{\eta}\overline{A}^{1-\sigma}\left(\frac{\Theta}{\theta}z\right)^{\gamma(1-\sigma)},\frac{\Theta}{\theta}z-\zeta\right\}
  \end{equation*}
  (which is clearly continuous) is such that
  \begin{equation*}
    \left\vert H_1(x,h,\mu,p)\right\vert\leq g(M(\mu))(hp+h^\eta)
  \end{equation*}
  for every $(x,\mu)$.\\
  Eventually, from (\ref{eq:al_DH1_1}) one immediately gets
  \begin{equation}
    \label{eq:bound_DH_twosided}
    -\zeta h\leq D_pH_1(x,h,\mu,p) \leq \left(\zeta+2F(x,\mu)\right)h,
  \end{equation}
from which the claim follows setting
\begin{equation*}
  g_1(z)=\zeta+2\frac{\Theta}{\theta}z.
\end{equation*}
\end{proof}
\begin{lemma}
  \label{lem:DpH_2}
  The function  $D_pH_1$ is locally Lispchitz continuous for $M(\mu)$ away from $0$, in the following sense: for every $N\geq 1$ there exists a positive constant $C_N$ such that for every $\mu_1,\mu_2$ satisfying $\frac1N<M(\mu_1)\wedge M(\mu_2)\leq \left(M_2(\mu_1)\vee M_2(\mu_2)\right)^{\frac12}<N$, for every $h_1\leq h_2<N$, for every $x_1,x_2$ and for every $p_1,p_2$ we have
\begin{multline*}
  \left\vert D_pH_1(x_1,h_2,\mu_1,p_1)-D_pH_1(x_2,h_2,\mu_2,p_2)\right\vert\\ \leq C_N\left(\vert x_1-x_2\vert + \vert h_1-h_2\vert + \sW_1(\mu_1,\mu_2) + \vert p_1-p_2\vert\right).
\end{multline*}
\end{lemma}
\begin{proof}
  We begin looking at increments in $p$ only. For every $(x,h,\mu)$ the map $p\mapsto D_pH_1(x,h,\mu,p)$ is differentiable with respect to $p$, and its derivative is
  \begin{equation*}
    D^2_{pp}H_1(x,h,\mu,p)=
    \begin{cases}
      0&\text{ for }p< p_0\\
      -\frac{1}{1-\eta}f(h)F(x,\mu)p_0^{\frac{1}{1-\eta}}p^{-\frac{\eta-2}{\eta-1}}&\text{ for }p> p_0.
    \end{cases}
  \end{equation*}
  $p\mapsto D^2_{pp}H_1(x,h,\mu,p)$ is not continuous at $p_0$, but
  \begin{equation}
    \label{eq:D2_bdd}
    \left\vert D^2_{pp}H_1(x,h,\mu,p)\right\vert \leq \frac{1}{1-\eta}\frac{1}{1-\gamma}\underline{A}^{\sigma-1}f(h)^{2-\eta}F(x,\mu)^{2-\gamma(1-\sigma)},
  \end{equation}
  thus $p\mapsto D^2_{pp}H_1(x,h,\mu,p)$ is bounded uniformly for $(h,\mu)$ in bounded sets. Therefore, for every $x,h,\mu,p_1,p_2$,
  \begin{equation*}
    \left\vert D_pH_1(x,h,\mu,p_1)-D_pH_1(x,h,\mu,p_2)\right\vert\leq \frac{1}{1-\eta}\frac{1}{1-\gamma}\underline{A}^{\sigma-1}L_f^{2-\eta}h^{2-\eta}\left(\frac{\Theta}{\theta}M(\mu)\right)^{2-\gamma(1-\sigma)}\left\vert p_1-p_2\right\vert,
  \end{equation*}
  so that
\begin{equation*}
  p\mapsto D_pH_1(x,h,\mu,p)
\end{equation*}
is Lipschitz, uniformly in $x$ and locally uniformly in $(h,\mu)$.\\

Now we look at increments in $h$. Fix $x$, $p$ and $\mu$ and consider $h_1\neq h_2$. The situation in which at least one of $M(\mu), p, h_1, h_2$ equals zero should be dealt with separately; for example if $M(\mu)=0$ then Lipschitzianity is immediate. To include all these cases in a consistent formulation we set $p_0(x,h,\mu)=+\infty$ whenever $h=0$ or $M(\mu)=0$, and we apply similar conventions below.\\
Now set
\begin{equation*}
  h_0=h_0(x,\mu,p)=f^{-1}\left(\left(\frac{1-\gamma}{p}A(x)^{1-\sigma}F(x,\mu)^{\gamma(1-\sigma)-1}\right)^{\frac{1}{1-\eta}}\right)
\end{equation*}
if $p\neq 0$ and $M(\mu)>0$, and
\begin{equation*}
  h_0=+\infty
\end{equation*}
otherwise. It follows that $p>p_0(x,h,\mu)$ if and only if $h>h_0(x,\mu,p)$, and we have $p=p_0(x,\mu,h_0(x,\mu,p))$ whenever $p\neq 0$ and $M(\mu)>0$. By (\ref{eq:al_DH1_1}) and the definitions of $p_0$ and $h_0$ we have
\begin{equation*}
  D_pH_1(x,h,\mu,p)=-\zeta h+F(x,\mu)\left(f(h)-f(h_0)\right)\ind_{[h_0,+\infty)}(h).
\end{equation*}
Fix now $(x,\mu,p)$ and take $h_1,h_2\in\bR_+$. If $h_1\vee h_2<h_0(x,\mu,p)$ we clearly have
\begin{equation*}
  \left\vert D_pH_1(x,h_1,\mu,p)-D_pH_1(x,h_2,\mu,p)\right\vert\leq \zeta\left\vert h_1-h_2\right\vert;
\end{equation*}
if $h_1\wedge h_2\geq h_0(x,\mu,p)$, since $f$ is Lipschitz, we have
\begin{align*}
  \left\vert D_pH_1(x,h_1,\mu,p)-D_pH_1(x,h_2,\mu,p)\right\vert&\leq \zeta\left\vert h_1-h_2\right\vert + F(x,\mu)L_f\left\vert h_1-h_2\right\vert\\
                                                               &\leq \left(\zeta+\frac{\Theta}{\theta}M(\mu)L_f\right)\left\vert h_1-h_2\right\vert.
\end{align*}
If $h_1<h_0(x,\mu,p)\leq h_2$, since $f$ is increasing, we have
\begin{align*}
  \left\vert D_pH_1(x,h_1,\mu,p)-D_pH_1(x,h_2,\mu,p)\right\vert&\leq \zeta\left\vert h_1-h_2\right\vert + F(x,\mu)\left\vert f(h_2)-f(h_0)\right\vert\\
                                                               &\leq \zeta\left\vert h_1-h_2\right\vert + F(x,\mu)\left\vert f(h_2)-f(h_1)\right\vert\\
                                                               &\leq \left(\zeta+\frac{\Theta}{\theta}M(\mu)L_f\right)\left\vert h_1-h_2\right\vert.
\end{align*}
The reverse situation where $h_1$ and $h_2$ are interchanged is analogous. Thus for every fixed $x,\mu,p$ the function $h\mapsto D_pH_1(x,h,\mu,p)$ is Lipschitz, with Lipschitz constant bounded by $\zeta+\frac{\Theta}{\theta}M(\mu)L_f$ (which depends only on $\mu$).\\

We now fix $x,h,p$ and consider increments in $\mu$. Set
\begin{equation*}
  F_0(x,h,p):=\left(p\frac{1}{1-\gamma}A(x)^{\sigma-1}f(h)^{1-\eta}\right)^{\frac{1}{\gamma(1-\sigma)-1}}
\end{equation*}
if $p\neq 0$ and $h\neq 0$, and
\begin{equation*}
  F_0(x,h,p)=+\infty
\end{equation*}
otherwise. This implies that $p>p_0(x,h,\mu)$ if and only if $F(x,\mu)>F_0(x,h,p)$, and we have $F_0\left(x,h,p_0(x,h,\mu)\right)=F(x,\mu)$ whenever $p$ and $h$ are nonzero. Consider $\mu_1,\mu_2$ such that $0<\underline{M}:=M_2(\mu_1)\wedge M_2(\mu_2)$, and set also $\overline{M}_2:=M_2(\mu_1)\vee M_2(\mu_2)$, $\alpha=\frac{1-\gamma(1-\sigma)}{1-\eta}$ and $\beta=\frac{1-\sigma}{1-\eta}$. If $F_0(x,h,p)> F(x,\mu_1)\vee F(x,\mu_2)$ then
\begin{equation*}
  \left\vert D_pH_1(x,h,\mu_1,p)-D_pH_1(x,h,\mu_2,p)\right\vert=0.
\end{equation*}
If $F_0(x,h,p)\leq F(x,\mu_1)\wedge F(x,\mu_2)$ we get, using (\ref{eq:al_DH1_1}), Lemma~\ref{barh}, the fact that $p\mapsto p^{-\frac{1}{1-\eta}}$ is decreasing and $1-\alpha<1$,
\begin{align*}
  \big\vert D_pH_1&(x,h,\mu_1,p)
  -D_pH_1(x,h,\mu_2,p)\big\vert
  \\
  &\leq f(h)\left\vert F(x,\mu_1)-F(x,\mu_2)\right\vert + p^{-\frac{1}{1-\eta}}(1-\gamma)^{\frac{1}{1-\eta}}A(x)^{\beta}\left\vert F(x,\mu_1)^{1-\alpha}-F(x,\mu_2)^{1-\alpha}\right\vert
  \\
  &\leq f(h)\frac{\Theta}{\theta^2}\left(M_2(\mu_1)^{\frac12}
  \left(L_{\eta_1}+L_{\eta_2}\right)+\Theta\right)\sW_2(\mu_1,\mu_2)\\
  &\phantom{aaaaa}
  +p_0(x,h,\mu_1)^{-\frac{1}{1-\eta}}(1-\gamma)^{\frac{1}{1-\eta}}
  A(x)^{\beta}\left\vert 1-\alpha\right\vert \left(\frac{\theta}{\Theta}\underline{M}\right)^{-\alpha}\left\vert F(x,\mu_1)-F(x,\mu_2)\right\vert\\                                                               &\leq f(h)\frac{\Theta}{\theta^2}\left(M_2(\mu_1)^{\frac12}
  \left(L_{\eta_1}+L_{\eta_2}\right)+
  \Theta\right)\sW_2(\mu_1,\mu_2)\left(1+F(x,\mu_1)^{\alpha}\left\vert 1-\alpha\right\vert\left(\frac{\theta}{\Theta}
  \underline{M}\right)^{-\alpha}\right)\\                                                               &\leq L_fh\frac{\Theta}{\theta^2}\left(\overline{M}_2^{\frac12}
  \left(L_{\eta_1}+L_{\eta_2}\right)+\Theta\right)\sW_2(\mu_1,\mu_2)\left(1+\left\vert 1-\alpha\right\vert\left(\frac{\Theta^2}{\theta^2}\frac{\overline{M}_2^{\frac12}}{\underline{M}}\right)^{\alpha}\right)\\
\end{align*}
If instead $F(x,\mu_2)\leq F_0(x,h,p)< F(x,\mu_1)$, which is equivalent to $p_0(x,h,\mu_1)\leq p<p_0(x,h,\mu_2)$, we have
\begin{align*}
  \big\vert D_pH_1&(x,h,\mu_1,p)-D_pH_1(x,h,\mu_2,p)\big\vert\\
  &\leq f(h)F(x,\mu_1)p^{-\frac{1}{1-\eta}}\left\vert p^{\frac{1}{1-\eta}}-p_0(x,h,\mu_1)^{\frac{1}{1-\eta}}\right\vert\\
           &\leq \frac{\Theta}{\theta}M(\mu_1)p_0(x,h,\mu_1)^{-\frac{1}{1-\eta}}(1-\gamma)^{\frac{1}{1-\eta}}A(x)^\beta\left\vert F_0(x,h,p)^{\gamma(1-\sigma)-1}-F(x,\mu_1)^{\gamma(1-\sigma)-1}\right\vert\\
           &\leq \frac{\Theta}{\theta}M(\mu_1)p_0(x,h,\mu_1)^{-\frac{1}{1-\eta}}(1-\gamma)^{\frac{1}{1-\eta}}A(x)^\beta\left\vert F(x,\mu_2)^{\gamma(1-\sigma)-1}-F(x,\mu_1)^{\gamma(1-\sigma)-1}\right\vert\\
           &\leq \frac{\Theta}{\theta}M(\mu_1)p_0(x,h,\mu_1)^{-\frac{1}{1-\eta}}(1-\gamma)^{\frac{1}{1-\eta}}A(x)^\beta\\
           &\phantom{aaaaaa}\cdot\left\vert \gamma(1-\sigma)-1\right\vert \left(\frac{\theta}{\Theta}\underline{M}\right)^{\gamma(1-\sigma)-2}\left\vert F(x,\mu_1)-F(x,\mu_1)\right\vert\\
           &\leq \frac{\Theta}{\theta}M(\mu_1)f(h) F(x,\mu_1)^\alpha\left(1-\gamma(1-\sigma)\right)\\
           &\phantom{aaaaaa}\cdot\left(\frac{\theta}{\Theta}\underline{M}\right)^{\gamma(1-\sigma)-2}\frac{\Theta}{\theta^2}\left(M_2(\mu_1)^{\frac12}\left(L_{\eta_1}+L_{\eta_2}\right)+\Theta\right)\sW_2(\mu_1,\mu_2)\\
           &\leq \left(\frac{\Theta}{\theta}\overline{M}_2^{\frac12}\right)^{1+\alpha}L_f h(1-\gamma(1-\sigma))\left(\frac{\Theta}{\theta}\frac{1}{\underline{M}}\right)^{2-\gamma(1-\sigma)}\\
           &\phantom{aaaaaa}\cdot\frac{\Theta}{\theta^2}\left(M_2(\mu_1)^{\frac12}\left(L_{\eta_1}+L_{\eta_2}\right)+\Theta\right)\sW_2(\mu_1,\mu_2).
\end{align*}
Therefore the map
\begin{equation*}
  \mu\mapsto D_pH_1(x,h,m,p)
\end{equation*}
is locally Lipschitz away from $M(\mu)=0$ locally uniformly with respect to $h$ and uniformly with respect to $(x,p)$.\\

To conclude we examine the increments of $D_pH_1$ in $x$, so we fix $h,\mu,p$ and consider $x_1,x_2\in \bR$. If $h=0$ or $M(\mu)=0$ or $p\leq p_0(x_1,h,\mu)\wedge p_0(x_2,h,\mu)$, we clearly have
\begin{equation*}
  \left\vert D_pH_1(x_1,h,\mu,p)-D_pH_1(x_2,h,\mu,p)\right\vert=0.
\end{equation*}
Therefore we assume that $h>0$ and $M(\mu)>0$ and $p$ is larger than at least one of $p_0(x_1,h,\mu)$ and $p_0(x_2,h,\mu)$. If $p>p_0(x_1,h,\mu)\vee p_0(x_2,h,\mu)$ we have by Lemma~\ref{barh}
\begin{align}
\nonumber  \big\vert &D_pH_1(x_1,h,\mu,p)-D_pH_1(x_2,h,\mu,p)\big\vert\\
  \nonumber            &\leq f(h)\left\vert F(x_1,\mu)-F(x_2,\mu)\right\vert + p^{-\frac{1}{1-\eta}}(1-\gamma)^{\frac{1}{1-\eta}}\left\vert A(x_1)^\beta F(x_1,\mu)^{1-\alpha}-A(x_2)^\beta F(x_2,\mu)^{1-\alpha}\right\vert\\
\nonumber                     &\leq 2 L_f h L_\eta\frac{\Theta}{\theta^2}M(\mu)\vert x_1-x_2\vert\\
\label{eq:Delta_1}            &\phantom{aaaa} + \left\vert A(x_1)^\beta F(x_1,\mu)^{1-\alpha}-A(x_2)^\beta F(x_1,\mu)^{1-\alpha}\right\vert + \left\vert A(x_2)^\beta F(x_1,\mu)^{1-\alpha}-A(x_2)^\beta F(x_2,\mu)^{1-\alpha}\right\vert.
\end{align}
We then have different (although similar) bounds, depending on $\alpha$ and $\beta$. Notice that $0<\alpha\leq 1$ if and only if $\gamma\geq \frac12$, while $\alpha>1$ otherwise.\\
For the last term in (\ref{eq:Delta_1}) we have
\begin{equation*}
  \left\vert A(x_2)^\beta F(x_1,\mu)^{1-\alpha}-A(x_2)^\beta F(x_2,\mu)^{1-\alpha}\right\vert\leq \overline{A}^\beta\vert 1-\alpha\vert\left(\frac{\theta}{\Theta}M(\mu)\right)^{-\alpha}2L_\eta\frac{\Theta}{\theta^2}M(\mu)\vert x_1-x_2\vert.
\end{equation*}
Set then
\begin{equation*}
  \Delta_1:=\left\vert A(x_1)^\beta F(x_1,\mu)^{1-\alpha}-A(x_2)^\beta F(x_1,\mu)^{1-\alpha}\right\vert.
\end{equation*}
If $\gamma\geq\frac12$ and $\beta<1$ we have
\begin{align*}
  \Delta_1&\leq\left(\frac{\Theta}{\theta}M(\mu)\right)^{1-\alpha}\left\vert A(x_1)^\beta- A(x_2)^\beta\right\vert \\
        &\leq\left(\frac{\Theta}{\theta}M(\mu)\right)^{1-\alpha}\beta \underline{A}^{\beta-1}\left\vert A(x_1)-A(x_2)\right\vert\\
        &\leq \left(\frac{\Theta}{\theta}M(\mu)\right)^{1-\alpha}\beta \underline{A}^{\beta-1}L_A\vert x_1-x_2\vert.
\end{align*}
If $\gamma\geq\frac12$ and $\beta >1$ we get, by similar computations,
\begin{equation*}
  \Delta_1\leq \left(\frac{\Theta}{\theta}M(\mu)\right)^{1-\alpha}\beta \overline{A}^{\beta-1}L_A\vert x_1-x_2\vert;
\end{equation*}
if $\gamma<\frac12$ and $\beta<1$ we get
\begin{equation*}
  \Delta_1\leq \left(\frac{\theta}{\Theta}M(\mu)\right)^{1-\alpha}\beta\underline{A}^{\beta-1}L_A\vert x_1-x_2\vert;
\end{equation*}
finally if $\gamma<\frac12$ and $\beta\geq 1$ we get
\begin{equation*}
  \Delta_1\leq \left(\frac{\theta}{\Theta}M(\mu)\right)^{1-\alpha}\beta\overline{A}^{\beta-1}L_A\vert x_1-x_2\vert.
\end{equation*}

To conclude we need to examine the case when $p_0(x_1,h,\mu)< p\leq p_0(x_2,h,\mu)$. Set
\begin{equation*}
  G_0(h,p)=\frac{1}{1-\gamma}f(h)^{1-\eta},
\end{equation*}
so that
\begin{equation*}
G_0\left(h,p_0(x,h,\mu)\right)=A(x)^{1-\sigma}F(x,\mu)^{\gamma(1-\sigma)-1}    
\end{equation*}
and
\begin{equation*}
A(x_1)^{1-\sigma}F(x_1,\mu)^{\gamma(1-\sigma)-1}<G_0(h,p)\leq A(x_2)^{1-\sigma}F(x_2,\mu)^{\gamma(1-\sigma)-1}.    
\end{equation*}
 Then
\begin{align*}
  \big\vert D_pH_1(x_1,h,&\mu,p)-D_pH_1(x_2,h,\mu,p)\big\vert\\
  &=f(h)F(x_1,\mu)p^{-\frac{1}{1-\eta}}\left(p^{\frac{1}{1-\eta}}-p_0(x_1,h,\mu)^{\frac{1}{1-\eta}}\right)\\
            &\leq f(h)\frac{\Theta}{\theta}M(\mu)p_0(x_1,h,\mu)^{-\frac{1}{1-\eta}}(1-\gamma)^{\frac{1}{1-\eta}}f(h)^{-1}\\
            &\phantom{aaaa}\cdot\left(G_0(h,p)^{\frac{1}{1-\eta}}-\left(A(x_1)^{1-\sigma}F(x_1,\mu)^{\gamma(1-\sigma)-1}\right)^{\frac{1}{1-\eta}}\right)\\
            &=f(h)F(x_1,\mu)\left(A(x_1)^{1-\sigma}F(x_1,\mu)^{\gamma(1-\sigma)-1}\right)^{-\frac{1}{1-\eta}}\\
            &\phantom{aaaa}\cdot\left(G_0(h,p)^{\frac{1}{1-\eta}}-G_0\left(h,p_0(x_1,h,\mu)\right)^{\frac{1}{1-\eta}}\right)\\
            &\leq L_f h \frac{\Theta}{\theta}M(\mu)A(x_1)^{-\beta}F(x_1,\mu)^\alpha\left\vert A(x_1)^\beta F(x_1,\mu)^{-\alpha}-A(x_2)^\beta F(x_2,\mu)^{-\alpha}\right\vert\\
            &\leq L_f h \left(\frac{\Theta}{\theta} M(\mu)\right)^{1+\alpha}\underline{A}^{-\beta}\\
            &\phantom{aaaa}\cdot\left(F(x_1,\mu)^{-\alpha}\left\vert A(x_1)^\beta - A(x_2)^\beta\right\vert + A(x_2)^\beta\left\vert F(x_1,\mu)^{-\alpha}-F(x_2,\mu)^{-\alpha}\right\vert\right)\\
            &\leq L_f h \left(\frac{\Theta}{\theta}\right)^{1+2\alpha} M(\mu)\underline{A}^{-\beta}\left(\beta\max\left\{\overline{A}^{\beta-1},\underline{A}^{\beta-1}\right\}L_A+\overline{A}^\beta\alpha 2L_\eta\frac{\Theta^2}{\theta^3}\right)\vert x_1-x_2\vert.
\end{align*}
The result follows substituting, in all bounds found above, $\underline{M}$ with $\frac1N$, $\overline{M}_2^{\frac12}$ with $N$, $h$ with $N$ and $M(\mu)$ with $\frac1N$ if it is raised to a negative exponent and with $N$ otherwise.
\end{proof}

\subsection{The Fokker-Planck equation.}

In this section, $V$ is any function satisfying (\ref{eq:ineqhDV}).\\
         We endow $C([0,T];\sP_2)$ with the topology induced by the distance
\begin{equation*}
          d_{\infty,2}\left(\mu(\cdot),\nu(\cdot)\right):=\sup_{t\in[0,T]}\sW_2\left(\mu(t),\nu(t)\right).
        \end{equation*}

To investigate well-posedness of the Fokker-Planck equation
\begin{equation}
  \label{eq:FP}
  \begin{cases}
    \partial_t\mu(t)=\frac12\epsilon^2 D^2_{xx}\mu(t) + \frac12 \chi^2 D^2_{hh}\left(h^2 \mu(t)\right) - D_x\left(D_pH_0(D_xV(t,x,h))\mu(t)\right)\\
    \phantom{aaaaaaaaaaaaaaaaaaaaaaaaaaaaaaaa}-D_h\left(D_pH_1(x,h,\mu(t),D_hV(t,x,h))\mu(t)\right),\\
    \mu(0)=\mu_0
  \end{cases}
\end{equation}
we consider the associated McKean-Vlasov stochastic differential equation
\begin{equation}
  \label{eq:MKV}
  \begin{cases}
    \ud \tilde{x}(t)=D_pH_0(D_xV(t,\tilde{x}(t),\tilde{h}(t)))\ud t+\epsilon \ud W^1(t),\\
    \ud \tilde{h}(t)=D_pH_1(\tilde{x}(t),\tilde{h}(t),\rL_{(\tilde{x}(t),\tilde{h}(t))},D_hV(t,\tilde{x}(t),\tilde{h}(t)))\ud t+\chi \tilde{h}(t)\ud W^2(t),\\
    \tilde{x}(0)=x_0,\quad \tilde{h}(0)=h_0,\quad\rL_{(\tilde{x}(0),\tilde{h}(0))}=\mu_0.
  \end{cases}
\end{equation}
We stress the fact that in this section the initial condition $(x_0,h_0)$ is random, with law satisfying Assumption~\ref{ass:id}.
\begin{proposition}
  \label{prop:MKV}
  There exists $T>0$ such that we have existence of a strong solution $(\tilde{x}(\cdot),\tilde{h}(\cdot))$ to (\ref{eq:MKV}) on $[0,T]$, which is morevoer pathwise unique.
\end{proposition}
\begin{proof}
  We begin by fixing a large $T_1>0$ and $\mu(\cdot)\in C^{\frac12}\left([0,T_1];\sP_2\right)$ such that $\mu(0)=\mu_0$. This allows us to exploit the regularity given by Proposition~\ref{prop:MFGregularity}.\\
  We then consider the stochastic differential equation
    \begin{equation}
  \label{eq:MKV_2}
  \begin{cases}
    \ud \hat{x}(t)=D_pH_0(D_xV(t,\hat{x}(t),\hat{h}(t)))\ud t+\epsilon \ud W^1(t),\\
    \ud \hat{h}(t)=D_pH_1(\hat{x}(t),\hat{h}(t),\mu(t),D_hV(t,\hat{x}(t),\hat{h}(t)))\ud t+\chi \hat{h}(t)\ud W^2(t),\\
    \hat{x}(0)=x_0,\hat{h}(0)=h_0, \rL_{(x_0,h_0)}=\mu_0.
  \end{cases}
\end{equation}

  By (\ref{eq:ineqhDV}), $D_hV(t,x,h)$ is locally bounded with respect to $h$, uniformly in $(t,x)$. Therefore Assumption~\ref{ass:DH_0} and estimate (\ref{eq:D2_bdd}) imply that the map
\begin{equation}
  \label{eq:drift_MKV_1}
   (x,h)\mapsto
   \begin{pmatrix}
     D_pH_0\left(D_xV(t,x,h)\right)\\
     D_pH_1\left(x,h,\mu(t),D_hV(t,x,h)\right)
   \end{pmatrix}
 \end{equation}
 is locally Lipschitz uniformly with respect to $t$ in the following sense: for every $N\in\bN$ and every $(x_1,h_1),(x_2,h_2)\in \bR\times\left(\frac1N,N\right)$ there exists a constant $L_{H,N}$ such that
 \begin{multline*}
   \left\vert\begin{pmatrix}
     D_pH_0\left(D_xV(t,x_1,h_1)\right)-D_pH_0\left(D_xV(t,x_2,h_2)\right)\\
     D_pH_1\left(x_1,h_1,\mu(t),D_hV(t,x_1,h_1)\right)-D_pH_1\left(x_2,h_2,\mu(t),D_hV(t,x_2,h_2)\right)
   \end{pmatrix}\right\vert\\\leq L_{H,N}\left\vert(x_1-x_2,h_1-h_2)\right\vert.
\end{multline*}
Each constant $L_{H,N}$ depends on $\mu$ but is bounded uniformly over compact sets in $\sP_2$.\\
Thanks to Remark~\ref{rem:H0}, the map $(t,x,h)\mapsto D_pH_0\left(D_xV(t,x,h)\right)$ is bounded by a constant $\overline{B}$ independent of $V$, while $D_pH_1(x,h,\mu,D_hV(t,x,h))$ has linear growth in $(x,h)$ thanks to Lemma~\ref{lem:DpH}.\\

If $\overline{M}(\mu(\cdot))=0$ (so that all the $h$-marginals of $\mu(t)$ equal $\delta_0$) then equation (\ref{eq:MKV_2}) clearly has a unique solution. In particular $h(\cdot)$ is then a geometric Brownian motion starting at $h_0$, so that $h(t)=0$ $\bP$-a.s. for every $t$ if $h_0=0$ $\bP$-a.s.. Moreover if $h_0=0$ $\bP$-a.s. then the $\bP$-a.s. constant process $h(\cdot)= 0$ is a solution to the second equation in (\ref{eq:MKV_2}) for any $\mu(\cdot)$.\\

Suppose now that $h_0>0$ $\bP$-a.s. and define
  \begin{gather*}
    \hat{b}_1(t,z,y)= D_pH_0\left(D_xV(t,z,e^y)\right),\\
    \hat{b}_2(t,z,y)= D_pH_1\left(z,e^y,\mu(t),D_hV(t,z,e^y)\right).
  \end{gather*}
  The map $(z,y)\mapsto \left(\hat{b}_1(t,z,y),e^{-y}\hat{b}_2(t,z,y)\right)$ is locally Lipschitz in the usual sense, uniformly with respect to $t$. Morevoer, by the boundedness of $D_pH_0\left(D_xV(t,x,h)\right)$ and thanks to Lemma~\ref{lem:DpH}, it is bounded. Thus, by standard results on SDEs, for every random variable $(z_0,y_0)$ there exists a strong solution, pathwise unique with continuous paths almost surely, to
  \begin{equation}
    \label{eq:MKV_exp}
      \begin{cases}
        \ud z(t)=\hat{b}_1(t,z,y)\ud t+ \epsilon\ud W^1(t),\\
        \ud y(t)=e^{-y}\hat{b}_2(t,z,y)\ud t -\frac{\chi^2}{2}\ud t+\chi\ud W^2(t),\\
        t\in[0,T],\ z(0)=z_0,\ y(0)=y_0.
      \end{cases}
    \end{equation}
    By Ito formula, $\left(\hat{x}(t),\hat{h}(t)\right)=\left(z(t),e^{y(t)}\right)$ is then a strong solution solution on $[0,T]$ to (\ref{eq:MKV_2}) for $x_0=z_0$ and $h_0=e^{y_0}$.\\
Conversely, take any solution $(\hat{x}(\cdot),\hat{h}(\cdot))$ to (\ref{eq:MKV_2}) with $h_0>0$ $\bP$-a.s.; the map $t\mapsto(\hat{x}(t),\hat{h}(t))$ is almost surely continuous, so that setting $\tau=\inf\left\{t\geq 0\colon h(t)\leq 0\right\}$ we have $\tau>0$ almost surely. Then for $t> 0$ we have, by Ito formula and (\ref{eq:bound_DH_twosided}),
\begin{align*}
  \log \hat{h}(t)&=\log h_0+\int_{0}^t\frac{1}{\hat{h}(s)}D_pH_1\left(\hat{x}(s),\hat{h}(s),\mu(s),D_hV(s,\hat{x}(s),\hat{h}(s))\right)\ud s-\frac12\chi^2t+\chi W^2(t)\\
                &\geq \log h_0-\left(\zeta+\frac12\chi^2\right)t+\chi W^2(t),
\end{align*}
and this shows, by contradiction taking the limit as $t\to\tau^-$, that $\hat{h}(t)>0$ for every $t$ almost surely. Therefore $\hat{y}(t)=\log\hat{h}(t)$ is well defined for all times and, again by Ito formula, $(\hat{x}(t),\hat{y}(t))$ is easily seen to be a solution to (\ref{eq:MKV_exp}). Pathwise uniqueness of solutions to (\ref{eq:MKV_exp}) thus entails pathwise uniqueness of solutions to (\ref{eq:MKV_2}).\\

The equivalence between (\ref{eq:MKV_exp}) and (\ref{eq:MKV_2}) yields also the existence for every $p\geq 2$ of a function $C_{p,T}\colon\bR_+\to\bR_+$ such that
\begin{equation*}
  \bE\left[\sup_{s\in[0,T]}\vert \hat{x}(s)\vert^p+\sup_{s\in[0,T]}\vert {h}(s)\vert^p\right]\leq C_{p,T}\big(\overline{M}(\mu)\big)\left(1+\bE\left[\vert x_0\vert^p+h_0^p\right]\right)
\end{equation*}
for every $\R\times\bR_{++}$-valued random variable $(x_0,h_0)$ with finite $p$-th moment. This implies that $\bE\left[\int_0^t\hat{h}(s)^2\ud s\right]$ is finite for every $t$ (because we are assuming the initial conditions to have finite second moment), so that $\bE\left[\int_0^t\hat{h}(s)\ud W^2(s)\right]=0$ for every $t\in[0,T_1]$.\\

Pathwise uniqueness of solutions to (\ref{eq:MKV_2}) when $h_0=0$ $\bP$-a.s. is a simple consequence of local Lipschitzianity of $h\mapsto D_pH_1(x,h,\mu,p)$ and boundedness of $D_hV$ on $[0,T_1]\times\R\times\R_+$ (see the proof of Lemma~\ref{lem:DpH_2}).\\

If $0<\mu_0\left(\bR\times\{0\}\right)<1$ (i.e. if $0<\bP(h_0=0)<1$), we can consider the two conditional laws of $(\tilde{x}(\cdot),\tilde{h}(\cdot))$ with respect to the sets $\{h_0=0\}$ and $\{h_0> 0\}$. By the arguments above, the measure defined by
  \begin{equation*}
    \rL(A)=\bP\left((\hat{x}(\cdot),\hat{h}(\cdot))\in A \vert h_0=0\right)\bP(h_0=0)+\bP\left((\hat{x}(\cdot),\hat{h}(\cdot))\in A \vert h_0>0\right)\bP(h_0>0)
  \end{equation*}
  for every measurable $A\subset C([0,+\infty))$ gives a weak solution to (\ref{eq:MKV_2}); moreover the process $(\hat{x}(\cdot),\hat{h}(\cdot))$ with law $\rL$ is pathwise unique, therefore the solution is strong by the Yamada-Watanabe theorem.\\

This shows that, for every $\mu(\cdot)\in C^{\frac12}\left([0,T_1];\sP_2\right)$ such that $\mu(0)=\mu_0$ and every $\bR\times\bR_+$-valued random variable $(x_0,h_0)$ with law $\mu_0$, equation (\ref{eq:MKV_2}) has a unique strong solution on $[0,T_1]$.\\

Now fix a $\bR\times\bR_+$-valued random variable $(x_0,h_0)$ and let $\mu_0\in\sP_2$ be its law. We first assume that $\mu_0(\R\times\{0\})=0$, so that in particular $M(\mu_0)>0$. For $T>0$ (to be determined later) consider the map
\begin{gather*}
\Xi_{(x_0,h_0)}\colon C\left([0,T];\sP_2\right)\to C\left([0,T];\sP_2\right),\\
  \Xi_{(x_0,h_0)}(\mu(\cdot))(t)=\rL_{(\hat{x}(t),\hat{h}(t))}
\end{gather*}
where $(\hat{x}(t),\hat{h}(t))$ solves (\ref{eq:MKV_2}). Well-posedness of this map is a consequence of what we have shown just above.\\
For positive constants $K_1,K_2$ (also to be determined later) consider now the set
\begin{multline}
  \label{eq:QKK}
  \sQ_{K_1,K_2}=\left\{\mu(\cdot)\colon [0,T]\to\sP_2\text{ s.t. }\mu(0)=\mu_0,\phantom{\sup_{\substack{[T]\\s\neq t}}\frac{\sW_2}{\vert t\vert^{\frac12}}}\right.\\
  \left.\sup_{t\in[0,T]}P_2(\mu(t))\leq 2 K_1, \sup_{\substack{s,t\in[0,T]\\s\neq t}}\frac{\sW_2(\mu(s),\mu(t))}{\vert s-t\vert^{\frac12}}\leq K_2\right\}.
\end{multline}
Then $\sQ_{T,K_1,K_2}$ is a convex compact set in $C([0,T];\sP_2)$.\\
    Suppose $\mu(\cdot)\in\sQ_{K_1,K_2}$ and recall from Lemma~\ref{lem:DpH} that $g_1(z)=\zeta+2\frac{\Theta}{\theta}z$.\\
    Considering only the equation for $\hat{h}(\cdot)$, we have
    \begin{align*}
      \bE&\left[\sup_{t\in[0,T]} \hat{h}(t)^2\right]\\
      &\leq 3 \bE[h_0^2]+3Tg_1^2(\overline{M}(\mu(\cdot)))\int_0^T\bE\left[\sup_{s\in[0,t]}h(s)^2\right]\ud t+3\chi^2\bE\left[\sup_{t\in[0,T]}\left\vert\int_0^t h(s)\ud W(s)\right\vert^2\right]\\
                                                   &\leq 3\bE[h_0^2]+12\left(g_1^2(\overline{M}(\mu(\cdot)))T+\chi^2\right)\int_0^T\bE\left[\sup_{s\in[0,t]}h(s)^2\right]\ud t,
    \end{align*}
(where the constant $12$ comes from the Burkholder-Davis-Gundy inequality for $p=2$). Thus Gronwall's inequality implies
    \begin{equation*}
      \bE\left[\sup_{t\in[0,T]} \hat{h}(t)^2\right]\leq 3\bE[h_0^2]e^{12(g_1^2(\overline{M}(\mu(\cdot)))T^2+\chi^2T)}\leq 3\bE[h_0^2]e^{12(g_1^2(\sqrt{2K_1})T^2+\chi^2T)}.
    \end{equation*}
    Choose $K_1$ such that
    \begin{equation}
      \label{eq:K1}
      K_1>3\bE[x_0^2+h_0^2]+3\overline{B}^2+12\epsilon;
    \end{equation}
    then we have that
\begin{equation*}
  \bE\left[\sup_{t\in[0,T]} \hat{h}(t)^2\right]\leq K_1
\end{equation*}
provided we choose $T$ small enough, namely in the interval
\begin{equation}
  \label{eq:interval}
  \left[0,\frac{\sqrt{144\chi^4+48\left(\zeta+2\frac{\Theta}{\theta}\sqrt{2K_1}\right)^2\log\left(\frac{K_1}{3\bE\left[h_0^2\right]}\right)}-12\chi^2}{24\left(\zeta+2\frac{\Theta}{\theta}\sqrt{2K_1}\right)^2}\right].
\end{equation}
As the upper bound on $T$ depends on $\bE[h_0^2]$, we cannot repeat the argument on consecutive intervals; thus the solution is only local in time.\\

Regarding $\hat{x}$ we have that
\begin{equation*}
  \bE\left[\sup_{t\in[0,T]}\vert \hat{x}(t)\vert^2\right]\leq 3\bE\left[x_0^2\right]+3\overline{B}^2T^2+12\epsilon T<K_1.
\end{equation*}
Therefore $\sup_{t\in[0,T]}P_2\left(\Xi_{(x_0,h_0)}\left(\mu(\cdot)\right)\right)\leq 2K_1$ if $\mu(\cdot)\in\sQ_{K_1,K_2}$ with $K_1$ and $T$ as above.\\

Clearly, for every $s\leq t\in[0,T]$,
\begin{equation*}
  \bE\left[\left\vert \hat{x}(t)-\hat{x}(s)\right\vert^2\right]\leq 2\overline{B}^2 (t-s)^2 + 2\epsilon^2 (t-s).
\end{equation*}
By sublinearity of the drift vector in (\ref{eq:drift_MKV_1}) we have
\begin{align*}
  \bE\left[\left\vert \hat{h}(t)-\hat{h}(s)\right\vert^2\right]&\leq \left(2(t-s)^2\left(\zeta+\frac{\Theta}{\theta}\overline{M}(\mu(\cdot))\right)^2+8\chi^2(t-s)\right)\bE\left[\sup_{r\in[s,t]}\hat{h}(r)^2\right]\\
                                                               &\leq \left(2(t-s)^2\left(\zeta+\frac{\Theta}{\theta}\sqrt{2K_1}\right)^2+8\chi^2(t-s)\right)K_1,
\end{align*}
which shows that
\begin{align*}
  \sW_2\left(\Xi_{(x_0,h_0)}\left(\mu(\cdot)\right)(t),\Xi_{(x_0,h_0)}\left(\mu(\cdot)\right)(s)\right)\leq\sqrt{\bE\left[\left\vert \hat{x}(t)-\hat{x}(s)\right\vert^2+\left\vert \hat{h}(t)-\hat{h}(s)\right\vert^2\right]}\leq K_2 \sqrt{t-s}
\end{align*}
as long as $K_2$ satisfies
\begin{equation}
  \label{eq:K2}
  K_2\geq\sqrt{4T\left(\overline{B}^2+\left(\zeta+\frac{\Theta}{\theta}\sqrt{2K_1}\right)^2K_1\right) + 2\epsilon^2+8\chi^2K_1}.
\end{equation}

Therefore, with $K_1$, $T$ and $K_2$ as above we have that $\Xi_{(x_0,h_0)}$ maps $\sQ_{K_1,K_2}$ into itself.\\

Now, with $(x_0,h_0)$ fixed as above, let $\mu(\cdot),\nu(\cdot)\in\sQ_{K_1,K_2}$ and denote by $\left(\hat{x}^\mu,\hat{h}^\mu\right)$ and $\left(\hat{x}^\nu,\hat{h}^\nu\right)$ the corresponding solutions to (\ref{eq:MKV_2}), respectively. Since $\hat{h}^\mu(t)$ and $\hat{h}^\nu(t)$ are positive almost surely, by lower-semicontinuity of the map $t\mapsto \bE\left[\hat{h}(t)\right]$ we have that both $\underline{M}(\mu(\cdot))$ and $\underline{M}(\nu(\cdot))$ are strictly positive on $[0,T]$. Therefore choose $N\in\bN$ such that $N>K_1$ and $\min\left\{\underline{M}(\mu(\cdot)),\underline{M}(\nu(\cdot))\right\}>\frac1N$. Set moreover
\begin{equation*}
  \tau_N=\inf\left\{t> 0\colon \hat{x}^\mu(t)\vee\hat{x}^\nu(t)\geq N\text{ or }\hat{h}^\mu(t)\wedge\hat{h}^\nu(t)\leq\frac1N \text{ or }\hat{h}^\mu(t)\vee\hat{h}^\nu(t)\geq N\right\}.
\end{equation*}
Notice that $D_xV$ and $D_hV$ are locally Lipschitz in $(x,h)$, uniformly in $t$. Therefore, for $t\in\left[0,\tau_N\right]$, by Assumption~\ref{ass:DH_0} and Lemma~\ref{lem:DpH_2} there exists a positive constant $\hat{C}_N$, depending also on $V$, such that
\begin{align*}
  \bE&\left[\left\vert\hat{x}^\mu(t)-\hat{x}^\nu(t)\right\vert^2\right]\\
     &\leq \bE\left[\left\vert\int_0^t\left(D_pH_0\left(D_xV\left(s,\hat{x}^\mu(s),\hat{h}^\mu(s)\right)\right)-D_pH_0\left(D_xV\left(s,\hat{x}^\nu(s),\hat{h}^\nu(s)\right)\right)\right)\ud s\right\vert^2\right]\\
     &\leq \hat{C}_N\int_0^t\left(\bE\left[\left\vert \hat{x}^\mu(s)-\hat{x}^\nu(s)\right\vert^2\right]+\bE\left[\left\vert \hat{h}^\mu(s)-\hat{h}^\nu(s)\right\vert^2\right]\right)\ud s
\end{align*}
and
\begin{align*}
  \bE&\left[\left\vert \hat{h}^\mu(t)-\hat{h}^\nu(t)\right\vert^2\right]\\
     &\leq 2\bE\left[\left\vert \int_0^t\left( D_pH_1\left(\hat{x}^\mu(s),\hat{h}^\mu(s),\mu(s),D_hV\left(s,\hat{x}^\mu(s),\hat{h}^\mu(s)\right)\right)\right.\right.\phantom{\left\vert\int_0^t\right\vert^2}\right.\\
     &\left.\left.\left.\phantom{\int_0^t}-D_pH_1\left(\hat{x}^\nu(s),\hat{h}^\nu(s),\mu(s),D_hV\left(s,\hat{x}^\nu(s),\hat{h}^\nu(s)\right)\right)\right)\ud s\right\vert^2\right]\\
     &\phantom{===}+2\bE\left[\left\vert\int_0^t\left(\hat{h}^\mu(s)-\hat{h}^\nu(s)\right)\ud W(s)\right\vert^2\right]\\
     &\leq \hat{C}_N\int_0^t\left(\bE\left[\left\vert\hat{x}^\mu(s)-\hat{x}^\nu(s)\right\vert^2\right]+\bE\left[\left\vert\hat{h}^\mu(s)-\hat{h}^\nu(s)\right\vert^2\right]+\sW_2^2\left(\mu(s),\nu(s)\right)\right)\ud s\\
     &\phantom{===}+2\bE\left[\left\vert\int_0^t\left(\hat{h}^\mu(s)-\hat{h}^\nu(s)\right)\ud W(s)\right\vert^2\right].
\end{align*}
Applyingh the Burkholder-Davis-Gundy inequality to the stochastic term and Gronwall's inequality to the function
\begin{equation*}
\bE\left[\sup_{t\in[0,T\wedge\tau_N]}\left\vert\hat{x}^\mu(t)-\hat{x}^\nu(t)\right\vert^2+\sup_{t\in[0,T\wedge\tau_N]}\left\vert\hat{h}^\mu(t)-\hat{h}^\nu(t)\right\vert^2\right]
\end{equation*}
we obtain that
\begin{equation}
  \label{eq:XI_cont}
d_{\infty,2}^2\left(\Xi_{(x_0,h_0)}\left(\mu(\cdot)\right),\Xi_{(x_0,h_0)}\left(\nu(\cdot)\right)\right)\leq  e^{\hat{C}_NT}\hat{C}_N\int_0^{T}\sW_2^2\left(\mu(s),\nu(s)\right)\ud s
\end{equation}
on $\{\tau_N>T\}$. However, $\lim_{N\to+\infty}\bP\left(\tau_N>T\right)=1$. This follows directly from Chebishev's inequality for what concerns the paths of $\hat{x}^\mu\vee\hat{x}^\nu$ and $\hat{h}^\mu\vee\hat{h}^\nu$; regarding the truncation imposed by $\tau_N$ on $\hat{h}^\mu\wedge\hat{h}^\nu$ from below, we simply proceed as follows. Set $\hat{y}^\mu=\log\hat{h}^\mu$; then amost surely $\tau_N=\inf\left\{t\colon\left\vert \hat{y}^\mu\right\vert\geq \log N\right\}$. Now, again by Chebishev's inequality, the same boundedness argument that lead to (\ref{eq:MKV_exp}) implies existence of a constant $\hat{E}$, depending on $V$, $T$ and $K_1$, such that
\begin{equation*}
  \bP\left(\tau_N\leq T\right)\leq \frac{\hat{E}\bE\left[\left\vert h_0\right\vert^2\right]}{\log^2 N}.
\end{equation*}
This readily implies the claim.\\
Therefore, (\ref{eq:XI_cont}) implies that $\Xi_{(x_0,h_0)}$ is continuous.\\
From Schauder's fixed point theorem then follows the existence of a fixed point. Such fixed point is actually unique thanks to Gronwall's inequality, because (\ref{eq:XI_cont}) and the last argument about $\tau_N$ imply in this case that
\begin{equation*}
  \sup_{t\in[0,T]}\sW_2^2\left(\mu(t),\nu(t)\right)\leq  e^{\hat{C}_NT}\hat{C}_N\int_0^{T}\sup_{s\in[0,t]}\sW_2^2\left(\mu(s),\nu(s)\right)\ud t.
\end{equation*}

It remains to deal with the case in which $\mu_0\left(\bR\times\{0\}\right)>0$. If $\mu_0\left(\bR\times\{0\}\right)=1$, i.e. if $h_0=0$ almost surely, it was already noticed that the only solution to (\ref{eq:MKV_2}) is $\hat{h}(t)=0$ for every $t$ almost surely, whatever $\mu(\cdot)$ is chosen. Therefore it is clear that the only solution to (\ref{eq:MKV}) is $\left(\tilde{x}(t),\tilde{h}(t),\rL_{\tilde{x}(t),\tilde{h}(t)}\right)=\left(\tilde{x}(t),0,\rL_{\tilde{x}(t)}\otimes\delta_o\right)$, where $\tilde{x}$ is the unique solution to the first equation in (\ref{eq:MKV}) corresponding to $h(t)=0$ for every $t$ almost surely.\\
  If instead $0<\mu_0(\bR\times\{0\})<1$ then $M(\mu_0)>0$ and we can again condition on the sets $\{h_0=0\}$ and $\{h_0>0\}$. As the solution $(\tilde{x}(\cdot),\tilde{h}(\cdot))$ is uniquely determined on each set, well-posedness of (\ref{eq:MKV}) follows.
  \end{proof}

Now we prove well-posedness of the Fokker-Planck equation (\ref{eq:FP}). We begin with a lemma about solutions of the Kolmogorov equation dual to a linearized version of (\ref{eq:FP}).
\begin{lemma}
  \label{lem:kolm}
  Fix $\mu(\cdot)\in C^{\frac12}\left([0,T];\sP_2\right)$ and define
  \begin{equation*}
    e^\mu(t,x,h)=
    \begin{pmatrix}
      D_pH_0\left(D_xV(t,x,h)\right)\\
      D_pH_1\left(x,h,\mu(t),D_hV(t,x,h)\right).
    \end{pmatrix}
  \end{equation*}
For every $\bar{t}\in[0,T]$ and every $\phi\in C^2_c\left(\bR\times\bR_+\right)$ there exists $u\in C^{1,2}\left([0,\bar{t}]\times\bR\times\bR_+\right)$ that satisfies
\begin{equation}
  \label{eq:Kolm}
  \begin{cases}
    \frac{\partial u}{\partial t}(t,x,h)+\frac12\mathrm{Tr}\left[g(x,h)g^\ast(x,h)D^2u(t,x,h)\right]+e^\mu(t,x,h)Du(t,x,h)=0,\\
    u(\bar{t},x,h)=\phi(x,h).
  \end{cases}
\end{equation}
\end{lemma}
\begin{proof}
  To prove such results we first perform, as in the proof of Proposition 3.3, the same exponential change of variable $e^y=h$.
Once we do this, thanks to Remark 2.4 and Lemma 3.4, we know that
the linear operator associated to the resulting linear PDE generates an analytic semigroup in the weightes spaces $C_w(\R^2,\R)$. Such analyticity, thanks to the result of Chapter 3 of the book of Lunardi \cite{Lunardi}, the required regularity.
\end{proof}
\begin{theorem}\label{thm:FKPunique}
  Let $T>0$ belong to the interval given in (\ref{eq:interval}). There is a unique measure-valued solution to \eqref{eq:FP} on $[0,T]$, in the sense of Definition~\ref{def:sol}.
\end{theorem}
\begin{proof}
  By Ito formula, the law $\rL_{(\tilde{x}(\cdot),\tilde{h}(\cdot))}$ of any solution to (\ref{eq:MKV}) is a solution to (\ref{eq:FP}). Denote this solution by $\mu(\cdot)$.
Fix $\phi$ and consider $e^\mu$ as in Lemma~\ref{lem:kolm} for this particular $\mu(\cdot)$. Let now $u$ be a $C^{1,2}$ solution to (\ref{eq:Kolm}). Using $u$ as a test function, we obtain that for every $0\leq t_0\leq t\leq \bar{t}$ the equality
\begin{multline*}
  \int_{\bR\times\bR_+} u(t,x,h)\mu(t,\ud x,\ud h) - \int_{\bR\times\bR_+} u(t_0,x,h)\mu(t_0,\ud x,\ud h)\\=\int_{t_0}^t\int_{\bR\times\bR_+} \frac{\partial u}{\partial t}(r,x,h)\mu(r,\ud x,\ud h)\ud r \\+ \frac12\int_{t_0}^t\int_{\bR\times\bR_+} \frac12\mathrm{Tr}\left[g(x,h)g^\ast(x,h)D^2u(r,x,h)\right]\mu(r,\ud x,\ud h)\ud r\\ + \int_{t_0}^t\int_{\bR\times\bR_+}e^\mu(r,x,h)\cdot Du(r,x,h)\mu(r,\ud x,\ud h)
\end{multline*}
holds. But this implies that
\begin{equation*}
  \int_{\bR\times\bR_+}\phi(x,h)\mu(t,\ud x,\ud h)=\int_{\bR\times\bR_+}u(0,x,h)\mu_0(\ud x,\ud h);
\end{equation*}
letting $\phi$ vary in $C^2_c$ we uniquely determine $\mu(t)$. Therefore $\mu(\cdot)$ is the unique solution to the linear Fokker-Planck equation with drift $e^\mu$.\\
If now $\nu(\cdot)$ is another solution to (\ref{eq:FP}), let $\mathbf{x}=(x^\prime,h^\prime)$ solve
\begin{equation}
  \label{eq:MKV_3}
  \begin{cases}
    \ud \mathbf{x}(t)=e^\nu(t,\mathbf{x}(t))\ud t+g(\mathbf{x}(t))\ud \mathbf{W}(t),\\
    \mathbf{x}(0)=(x_0,h_0).
  \end{cases}
\end{equation}
    There exists only one such $\mathbf{x}$: indeed (\ref{eq:MKV_3}) is simply equation (\ref{eq:MKV_2}) with $\nu$ in place of $\mu$. Therefore, by Ito formula, its law $\rL_{\mathbf{x}}$ is a solution to the linear Fokker-Planck equation
    \begin{equation*}
      \frac{\partial \eta}{\partial t}(t)=\frac12\sum_{i=x,h} D^2_{i,i}\left(g g^\ast \eta(t)\right)-\mathrm{div}\left(e^\nu \eta(t)\right).
    \end{equation*}
    This yields, by uniqueness as shown above, that $\rL_{\mathbf{x}(t)}=\nu(t)$ for every $t\in[0,T]$. In turn, this implies that $(\mathbf{x},\nu)$ solve the McKean-Vlasov equation (\ref{eq:MKV}); but solutions to the latter are unique, therefore $\nu(t)=\mu(t)$ and uniqueness is proven.
\end{proof}

\section{Solution to the Mean Field Game}\label{sec:MFG}
We will need the following simple lemma.
	\begin{lemma}\label{lem:hmoment}
Let $\mu(\cdot) \in C([0,T], \mathcal{P}_2)$ and $(s(\cdot),v(\cdot))\in\sK$ be fixed. Then the solution $(x(\cdot)h(\cdot))$ to (\ref{eq:evolxnew})-(\ref{eq:evolhnew1}) with initial condition $(x(t_0),h(t_0))=(x_0,y_0)$ having law $\mu_0$ satisfies for every $t\in [t_0,T]$
\begin{equation}\label{eqn:stimamediah}
 \bE\left[h(t)\right]\leq  2 e^{\frac{C_{2,2}}{2}(T-t_0)}\bE \left[h_0^2\right]^{\frac{1}{2}}.
\end{equation}
Moreover if $\nu\in C([0,T];\sP_2)$ and we denote by $(x_\mu(\cdot),h_\mu(\cdot))$ and $(x_\nu(\cdot),h_\nu(\cdot))$ the solutions to (\ref{eq:evolxnew})-(\ref{eq:evolhnew1}) with initial condition $(x(t_0),h(t_0))=(x_0,y_0)$ having law $\mu_0$ and drift evaluated along $\mu(\cdot)$ and $\nu(\cdot)$, respectively, we have
\begin{equation}
  \label{eq:delta_h_mu}
  \bE\left[\sup_{t\in[0,T]}\left\vert h_\mu(t)-h_\nu(t)\right\vert^2\right]\leq C(\mu,\nu,\mu_0,T) d_{\infty,2}(\mu(\cdot),\nu(\cdot))^2;
\end{equation}
where
\begin{multline}
  \label{eq:C_long}
  C(\mu,\nu,\mu_0,T)=24T^2 \left(\frac{\Theta}{\theta^2}\left(L_{\eta_1}+L_{\eta_2}^{-1}\right)\max\left\{P_2(\mu)^{\frac12},P_2(\nu)^{\frac12}\right\}+\frac{\Theta^2}{\theta^2}\right)^2\\\cdot\bE\left[h_0^2\right]e^{TC_{2,2}(T,\nu(\cdot))+6TL_f^2\frac{\Theta^2}{\theta^2}M(\mu)^2+3T\zeta^2+3\chi^2}.
\end{multline}
In particular, for fixed $T$ and $\mu_0$, $C(\mu,\nu,\mu_0,T)$ is bounded uniformly for $\mu(\cdot),\nu(\cdot)$ in compact sets of $C([0,T];\sP_2)$.
\end{lemma}
\begin{proof}
Estimate \eqref{eqn:stimamediah} follows by H\"older's inequality applied to the second estimate in (\ref{eq:est_h_1}). Estimate (\ref{eq:delta_h_mu}) follows from the Lipschitz property of $f$ and $F$ (see Lemma~\ref{barh}), from (\ref{eq:est_h_1}) and Gronwall's inequality.
\end{proof}

\begin{lemma}\label{lem:locgradV}
Let $\mu_n(\cdot) \to \mu(\cdot)$ in  $C([0,T], \mathcal{P}_2)$ and let $V_n, V$ be the corresponding value functions defined via \eqref{eqn:V}. Then for all $r>0$ and all $n\in\bN$ it holds
\begin{equation}\label{eq:Vn}
\sup_{[0,T]\times (-r,r)\times (0,r)}\left\vert V_n\right\vert\leq C(T, r,\mu)
\end{equation}

\begin{equation}\label{eq:Dvn}
\sup_{[0,T]\times  (-r,r)\times(0,r)}\left(|D_x V_n|+|D_hV_n|\right)\leq C(T, r,\mu),
\end{equation}
\begin{equation}\label{eq:DVbetan}
\left(|D_x V_n|_{[0,T]\times (-r,r)\times (0,r)}^{(\beta)}+|D_h V_n|_{[0,T]\times (-r,r)\times (0,r)}^{(\beta)}\right)\leq C(T, r,\mu),
\end{equation}
where $|\cdot|_{[0,T]\times (-r,r)\times (0,r)}^{(\beta)}$ denotes the H\"older seminorm of exponent $\beta$ in $[0,T]\times (-r,r)\times (0,r)$, $\beta>0$, and $C(T, r, \mu)$ depends on $T,r,\overline{M}(\mu(\cdot))$ and is independent of $n$.
\end{lemma}

\begin{proof}
  We will use the fact that the value function is a viscosity solution to the HJB equation and apply standard estimates available in the literature. The uniform estimates for the derivative of $V_n$ with respect to $x$ follow immediately from Theorem $3.1$ of \cite{LadyP}, Chapter V with $Q_T=[0,T]\times (-r-1,r+1)\times (0,r+1)$ and $Q_T'=[0,T]\times (-r,r)\times (0,r)$.\\
  However, for the derivatives with respect to $h$ the available estimates hold for uniformly parabolic equations, whereas our Hamilton-Jacobi-Bellman equation degenerates in $h=0$. To overcome this issue, we apply the change of variable
\begin{equation}\label{eq:change}
y=\log h, \quad h>0.
\end{equation}
Then the value function $V_n(t,x,e^y):=W_n(t,x,y)$ is a solution to the Hamilton-Jacobi-Bellman equation
\begin{equation}\label{MFGW}
-\partial_t W_n+\rho W_n=\tilde H_1(t,x,y,\mu_n(t), D_yW_n)+\frac{1}{2}\chi^2D^2_{yy}W_n+H_0(D_xW_n)
\end{equation}
 in $(0,T)\times \bR\times \R$, where
\begin{equation*}
\tilde H_1(x,y,\mu(t), D_yW)=\sup_{ s\in [0,1]}\left\{B_1(s,  x, y, \mu(t))D_{y} W+ u_\sigma\left( B_2(s, x, y, \mu(t))\right)\right\},
\end{equation*}
with
\begin{equation*}
	B_1(s,x, y,\mu(t))=sf(e^y)e^{-y} F(x,\mu(t)) -\left(\zeta+\frac{\chi^2}{2}\right)
      \end{equation*}
and
\begin{equation*}
	B_2(s,  x, y, \mu(t))=A(x)(1-s)^{1-\gamma}f(e^y)^{1-\gamma} F(x,\mu(t))^\gamma
      \end{equation*}
        Once we have proved the analogues of \eqref{eq:Vn}, \eqref{eq:Dvn} and \eqref{eq:DVbetan} for $W_n$ and its derivatives with respect to $h$, via the change of variable \eqref{eq:change} in the estimate we can easily conclude \eqref{eq:Vn}, \eqref{eq:Dvn} and \eqref{eq:DVbetan} for $V_n$.\\

Similarly to Proposition \ref{prop:MFGregularity} with $\mu(t)=\mu_n(t)$ one can prove that $W_n \in C^{1,2}((0,T)\times \R\times \R)$. Then we apply again Theorem $3.1$ of \cite{LadyP}, Chapter V with $Q_T$ and $Q_T'$ as above (note that the Hamilton-Jacobi-Bellman in \eqref{MFGW} is uniformly parabolic).  To apply the theorem we exploit the estimates proved in Lemma \ref{lem:DpH}.\\
Moreover, the estimates of Theorem $3.1$ of \cite{LadyP}  depend  continously on $g(\mu_n), g_1(\mu_n)$ (defined in Lemma~\ref{lem:DpH}), and on $\sup_{Q_T}|W_n|$. Since $g, g_1$ are continuous in $M(\mu)$ and $M(\mu_n(t))\to M( \mu(t))$ uniformly with respect to $t$, for $n$ large enough $g(\mu_n)$ can be estimated by a costant independent on $n$.\\
Eventually we estimate  $\sup_{Q_T}|W_n|$. Recalling that $V_n$ is the value function as defined in (\ref{eqn:V}) for the functional $J$ given by (\ref{eqn:J}), using the assumptions on $f$ and on $F$, by Lemma \ref{barh}, Jensen's inequality and Lemma \ref{lem:hmoment} we have
\begin{eqnarray*}
|W_n(t_0, x_0, y_0)|=|V_n(t_0,x_0,e^{y_0})|&\leq&  C\overline{M}(\mu_n(\cdot))^{\gamma(1-\sigma)}\int_{t_0}^T e^{-\rho t}\mathbb{E}\left[h_n(t)^\eta\right]\, dt\\&\leq &C\overline{M}(\mu_n(\cdot))^{\gamma(1-\sigma)}\int_{t_0}^T e^{-\rho t}\mathbb{E}\left[h_n(t)\right]^\eta\, dt \\ &\leq& C\overline{M}(\mu_n(\cdot))^{\gamma(1-\sigma)}\int_{t_0}^T e^{-\rho t}2^\eta e^{\frac{\eta C_{2,2}}{2}(T-t_0)}\bE \left[(e^{y_0})^2\right]^{\frac{\eta}{2}}\ud t.
\end{eqnarray*}
Similarly as above, since $\mu_n(\cdot) \to \mu(\cdot)$ in $C([0,T], \mathcal{P}_2)$, then $M(\mu_n(t))\to M(\mu(t))$ for all $t \in [0,T]$ and therefore, for $n$ large enough, we can estimate $|W_n|$ by a constant idependent of $n$.
\end{proof}
\begin{remark}
  Estimates (\ref{eq:Vn}) and (\ref{eq:Dvn}) can also be proved directly, with computations very similar to those in the proofs of Propositions~\ref{prop:MFGregularity} and \ref{prop:VFestimates}, thanks to the fact that for every fixed sequence $\mu_n(\cdot)$ one can bound $\overline{M}(\mu_n(\cdot))$ with a constant independent of $n$.
\end{remark}

We can now prove existence of a solution to the mean-field game system \eqref{MFGb}.\\

	\begin{theorem}\label{thm:existence}
		Let $T>0$ belong to the interval given in (\ref{eq:interval}). Under the standing assumptions, there exists a solution $(V,\mu)$ of \eqref{MFGb} on $[0,T]$ in the sense of Definition~\ref{def:sol}, where $T$ is given
	\end{theorem}

\begin{proof}
A general scheme to prove this type of results has been given for example in \cite{cardnote}.
          We aim to apply Schauder's fixed point theorem; the proof is divided in several steps.\\

        \emph{Step 1.}  As the solution of the Hamilton-Jacobi equation \eqref{eqn:HJB} is not unique, we choose the value function among all possible solutions.
        Doing so, we can define the map $\Psi_{\mu_0}$ (to which we will apply Schauder's fixed point theorem to deduce existence of a fixed point in a suitable subset of $C([0,T],\mathcal{P}_2)$) in the following way.\\
        Given $\mu(\cdot) \in \mathcal{C}$ we consider the function $V$ as defined in \eqref{eqn:V}. Thanks to Proposition \ref{prop:MFGregularity} $V$ is a smooth solution to \eqref{eqn:HJB}. Then we define
        \begin{gather*}
          \Psi_{\mu_0}\colon C([0,T],\mathcal{P}_2)\to C([0,T],\mathcal{P}_2),
        \Psi_{\mu_0}(\mu(\cdot))=m(\cdot)
      \end{gather*}
      where $m(\cdot)$ is the unique solution to the Fokker-Planck equation \eqref{eq:FP}, as provided by Theorem~\ref{thm:FKPunique}.\\
    \emph{Step 2.} By the results of the previous section, $\Psi$ is well-defined. For $K_1$ and $K_2$ as in~(\ref{eq:K1})-~(\ref{eq:K2}), $\sQ_{K_1,K_2}$ as defined in (\ref{eq:QKK}) is a compact convex set in $C([0,T];\sP_2)$ that is left invariant by $\Psi_{\mu_0}$. Indeed, the only difference between $\Psi_{\mu_0}$ and $\Xi_{(x_0,h_0)}$ is that the latter was defined for a generic function $V$ satisfying certain bounds, while in the former we choose a specific $V$, namely the value function of the optimization problem. Therefore the fact that $\sQ_{K_1,K_2}$ is invariant under the action of $\Psi_{\mu_0}$ can be proved in exactly the same way as the analogous claim for $\Xi_{(x_0,h_0)}$ in the proof of Proposition~\ref{prop:MKV}.\\

        \emph{Step 3.} Now we prove that $\Psi_{\mu_0}$ is continuous with respect to the topology induced by the distance $d_{\infty,2}$.\\
        Let $\mu_n(\cdot)\in\sQ_{K_1,K_2}$ be a convergent sequence with respect to $d_{\infty,2}$, and denote by $\mu(\cdot)$ its limit point. Let moreover $V_n, V_\mu$ be the value functions of the maximization problem for (\ref{eqn:J}) associated to $\mu_n(\cdot)$ and $\mu(\cdot)$, respectively.\\
        First we prove that  $V_n(t_0,\bar x_0, \bar h_0) \to V(t_0, \bar x_0, \bar h_0)$ for all $t_0 \in [0,T], \bar x_0 \in \R, \bar h_0 \in \R_+$; we will write $\bar\delta_0$ for $\delta_{\bar x_0,\bar h_0}$ (the Dirac measure with mass $1$ in $(\bar x_0,\bar h_0)$ and mass $0$ everywhere else). For all $n \in \mathbb{N}$ and $\eps >0$, let $(v_n, s_n) \in \mathcal{K}$ be $\eps$-optimal controls for $V_n(t_0,x_0,h_0)$, denote by $\left(x_n(\cdot),h_n(\cdot)\right)$ the solution to
\begin{equation*}
\begin{cases}
  dx_n(t) = v_n(t) dt  +\epsilon dZ(t)\quad\text{ for }t\in[t_0,T],\\
    dh_n(t) = s_n(t) f(h_n(t)) F(x_n(t), \mu_n(t)) -\zeta h_n(t)dt +\chi h_n (t) dW(t)\quad\text{ for }t\in[t_0,T],\\
      x_n(t_0)=\bar x_{0}, h_n(t_0)=\bar h_{0}
  \end{cases}
\end{equation*}
and by $\left(x_\mu(\cdot),h_\mu(\cdot)\right)$ the solution to
\begin{equation*}
  \begin{cases}
  dx_\mu(t) = v_n(t) dt  +\epsilon dZ(t)\quad\text{ for }t\in[t_0,T],\\
    dh_\mu(t) = s_n(t) f(h_\mu(t)) F(x_\mu(t), \mu(t)) -\zeta h_\mu(t)dt +\chi h_\mu (t) dW(t)\quad\text{ for }t\in[t_0,T],\\
      x_\mu(t_0)=\bar x_{0}, h_\mu(t_0)=\bar h_{0}.
  \end{cases}
\end{equation*}
Actually $x_n(t)=x_\mu(t)$ almost surely for every $t$, since the measures $\mu$ and $\mu_n$ only affect the dynamics via the equation for $h$. We have
\begin{multline*}
  \left\vert V_n(t_0,\bar x_0,\bar h_0)-V_\mu(t_0,\bar x_0,\bar h_0)\right\vert\leq \\
  \eps+\mathbb{E}\left[\int_{t_0}^{T} e^{-\rho t} \left\vert J_n(x_n(t),h_n(t), s_n(t), v_n(t))-J_\mu(x_n(t),h_\mu(t), s_n(t), v_n(t))\right\vert\ud t\right],
\end{multline*}
where for simplicity we have set
\begin{equation*}
J_n(x_n(t),h_n(t), s_n(t), v_n(t))=u_\sigma((1-s_n(t))^{1-\gamma}f(h_n(t))^{(1-\gamma)}F(x_n(t),\mu_n(t))^\gamma A(x_n(t)))
\end{equation*}
and
\begin{equation*}
		J_\mu(x_n(t),h_\mu(t), s_n(t), v_n(t))=u_\sigma((1-s_n(\tau))^{1-\gamma}f(h_\mu(t))^{(1-\gamma)}F(x_n(t),\mu(t))^\gamma A(x_n(t))).
              \end{equation*}
              By Lemma \ref{barh}, there exists a constant $\hat{C}$ independent of $n$ ($\hat{C}$ depends only on $K_1,K_2$ and $T$) such that
\begin{align*}
  &\left\vert J_n(x_n(t),h_n(t), s_n(t), v_n(t))-J_n(x_n(t),h_\mu(t), s_n(t), v_n(t))\right\vert\\
             &\leq \hat{C}\left(f(h_n(t))^\eta \left\vert F(x_n(t),\mu_n(t))-F(x_n(t),\mu(t))\right\vert^\gamma+F(x_n(t), \mu(t))^\gamma\left\vert f(h_n(t))-f(h_\mu(t))\right\vert^\eta\right)\\
             &\leq \hat{C}\left(h_n(t)^\eta d_{\infty,2}(\mu_n, \mu)^\gamma+\left\vert h_n(t)-h_\mu(t)\right\vert^\eta\right).
\end{align*}
Thus by Jensen's inequality, Lemma \ref{lem:hmoment} and (\ref{eq:est_h_1}) we have
\begin{align*}
  &\left\vert V_n(t_0,x_0,h_0)-V_\mu(t_0,x_0,h_0)\right\vert\\
  &\leq \eps+\hat{C}d_{\infty,2}(\mu_n(\cdot), \mu(\cdot))^\gamma\int_{t_0}^{T}\mathbb{E}[h_n(t)^2]^{\frac{\eta}{2}}  dt+\hat{C}\int_{t_0}^T \mathbb{E}[\left\vert h_n(t)-h_\mu(t)\right\vert^2]^{\frac{\eta}{2}} dt\\
  &\leq  \eps+4\hat{C} d_{\infty,2}(\mu_n(\cdot), \mu(\cdot))^\gamma (T-t_0)e^{(T-t_0)C_{2,2}(T,\mu_n(\cdot))}\bar h_0^\eta+\hat{C} d_{\infty,2}(\mu_n(\cdot), \mu(\cdot))^{\eta}C(\mu_n,\mu,\bar\delta_0,T)^{\frac{\eta}{2}},
\end{align*}
where $C_{2,2}(T,\mu_n(\cdot)$ is given in (\ref{eq:constants_C_2}) and $C(\mu_n,\mu,\bar\delta_0,T)$ in (\ref{eq:C_long}). The sequences $\{C_{2,2}(T,\mu_n(\cdot))\}$ and $\{C(\mu_n,\mu,\bar\delta_0,T)\}$ are bounded (actually since the Wasserstein distance $\sW_2$ metrizes weak convergence together with convergence of second moments, we also have that
\begin{equation*}
  C_{2,2}(T,\mu_n(\cdot))\overset{n\to+\infty}{\longrightarrow}C_{2,2}(T,\mu(\cdot))\quad\text{ and }\quad C(\mu_n,\mu,\bar\delta_0,T)\overset{n\to+\infty}{\longrightarrow} C(\mu,\mu,\bar\delta_0,T),
\end{equation*}
and the same conclusion actually holds for every measure $\nu\in\sP_2$ in place of $\bar \delta_0$). By the arbitrariness of $\eps$ we deduce the convergence of $V_n$ to $V$, locally uniformly on $[0,T]\times\bR\times\bR_+$.\\

By Lemma \ref{lem:locgradV}, $D_hV_n$ is uniformly bounded in $n$ and uniformly H\"older continuous in $n$ in every compact set of $[0,T] \times \R\times \R_+$. Set $K_1=[0,T]\times (-r,)]\times (0,r)$ for ome fixed $r>0$. By the Ascoli-Arzel\`a theorem we can extract a subsequence $\left\{D_hV_{n}^1\right\}$ of $\left\{D_nV_n\right\}$ which converges locally uniformly on $K_1$. Similarly, we set $K_2=[0,T]\times (-r-1,r+1)\times (0,r+1)$ and we can extract a subsequence $\left\{D_h V_{n}^2\right\}$ of $\left\{D_h V_{n}^1\right\}$ that converges locally uniformly $K_2$. We repeat the same argument on each relatively compact set $K_k=[0,T]\times (-r-k+1,r+k-1)\times (0, r+k-1)$ and deduce the existence, for each $k$, of a subsequence $\left\{D_h V_{n}^k\right\}$ uniformly converging in $K_k$. Then, the diagonal subsequence $\left\{D_hV_{m}^m\right\}$ converges locally uniformly to some function $Q$. But also $V_n$ converges to $V$ locally uniformly, so that by the regularity of $V_n$ and $V$ we must have $Q=D_hV$. Since this construction can be applied to every subsequence of $\{D_hV_n\}$, also the original sequence $\{D_hV_n\}$ converges locally uniformly to $D_hV$. The exact same argument can be applied to $\{D_xV_n\}$, yielding its locally uniform convergence to $D_xV$. By the arbitrariness of $(t_0,\bar x_0,\bar h_0)$ we deduce the convergence of $V_n$, $D_xV_n$ and $D_hV_n$ on $[0,T]\times\bR\times\bR_{++}$.\\
Now define
\begin{equation*}
 m_n(\cdot):=\Psi_{\mu_0}(\mu_n(\cdot)),\quad m(\cdot):=\Psi_{\mu_0}(\mu(\cdot));
\end{equation*}
each $m_n(\cdot)$ solve a Fokker-Planck equation that is the same as the second equation in (\ref{MFGb}), with coefficients evaluated in $DV_n$. Consider now a subsequence of $\{m_n(\cdot)\}$, whose elements we still denote by $m_n(\cdot)$; such subsequence belongs to $\sQ_{K_1,K_2}$, thus it has a convergent subsequence $m_{n_k}(\cdot)$. Let $\bar{m}(\cdot)$ be its limit point and fix $\phi\in C_c^{\infty}([0,T]\times\bR\times\bR_+)$; we will show that the Fokker-Planck equation for $m_{n_k}(\cdot)$, in the sense of (\ref{eq:sol_weak}) with coefficients evaluated ind $DV_n$, converges term by term to the Fokker-Planck equation (\ref{eq:sol_weak}) for $\bar{\mu}(\cdot)$ (with coefficients evaluated in $DV$). Since the solution to the latter is unique, this proves that every subsequence of $m_n(\cdot)$ has a subsequence which converges to the same limit $\bar{\mu}(\cdot)$, so the whole sequence $m_n(\cdot)$ converges to $\bar{m}(\cdot)$ and we have, again by uniqueness, $\bar{m}(\cdot)=m(\cdot)$, hence yielding continuity of $\Psi_{\mu_0}$ on $\sQ_{K_1,K_2}$.\\
Since $m_{n_k}(t)$ converges weakly to $\bar{m}(t)$ for every $t$, we have that
\begin{equation*}
  \int_{\bR\times\bR_+}\phi(t,x,h)m_{n_k}(t;\ud x,\ud h)\to\int_{\bR\times\bR_+}\phi(t,x,h)\bar{m}(t;\ud x,\ud h).
\end{equation*}
We then have, for every $r\in[0,T]$,
\begin{align*}
  \int_{\bR\times\bR_+}D_pH_0&(D_xV_{n_k}(r,x,h))D_x\phi(r,x,h)m_{n_k}(r;\ud x,\ud h)\\
                             &\phantom{\leq}-\int_{\bR\times\bR_+}D_pH_0(D_xV(r,x,h))D_h\phi(r,x,h)\bar{m}(r;\ud x,\ud h)\\
                             &\leq \overline{B} \int_{\bR\times \bR_+}D_x\phi(r,x,h)\left(m_{n_k}(r;\ud x,\ud h)-\bar{m}(r;\ud x,\ud h)\right) \\
                             &\phantom{\leq}+ \int_{\bR\times\bR_+} \left(D_pH_0(D_xV_{n_k}(r,x,h))-D_pH_0(D_xV(r,x,h))\right)\bar{m}(r;\ud x,\ud h)\\
                             &\leq \overline{B} \int_{\bR\times \bR_+}D_x\phi(r,x,h)\left(m_{n_k}(r;\ud x,\ud h)-\bar{m}(r;\ud x,\ud h)\right) \\
                             &\phantom{\leq}+ \hat{L}\int_{\bR\times\bR_+} \left\vert D_xV_{n_k}(r,x,h)-D_xV(r,x,h)\right\vert\bar{m}(r;\ud x,\ud h)
\end{align*}
for some constant $\hat{L}$, because $D_xV_n$ is bounded uniformly in $n$ and $D_pH_0$ is locally Lipschitz. The right hand side in the previous inequality then converges to $0$ because $m_{n_k}$ converges weakly to $\bar{m}$ and $D_xV_{n_k}$ converges locally uniformly to $D_xV$.\\
The terms
\begin{equation*}
  \int_{\bR\times\bR_+}\frac{\partial\phi}{\partial t}(r,x,h)m_{n_k}(r;\ud x,\ud h)-\int_{\bR\times\bR_+}\frac{\partial\phi}{\partial t}(r,x,h)\bar{m}(r;\ud x,\ud h)
\end{equation*}
and
\begin{equation*}
  \int_{\bR\times\bR_+}D^2_{xx}\phi(r,x,h)m_{n_k}(r;\ud x,\ud h)-\int_{\bR\times\bR_+}D^2_{xx}\phi(r,x,h)\bar{m}(r;\ud x,\ud h)
\end{equation*}
trivially converge to $0$ since $m_{n_k}$ converges to $\bar{m}$ weakly.\\
For the same reason, together with $D^2_{hh}\phi$ having compact support, also the term
\begin{equation*}
  \int_{\bR\times\bR_+}h^2D^2_{hh}\phi(r,x,h)m_{n_k}(r;\ud x,\ud h)-\int_{\bR\times\bR_+}h^2D^2_{hh}\phi(r,x,h)\bar{m}(r;\ud x,\ud h)
\end{equation*}
converges to $0$.\\
For the remaining term we have
\begin{align}
  \nonumber  &\int_{\bR\times\bR_+}D_pH_1(x,h,m_{n_k}(r),D_hV_{n_k}(r,x,h))D_h\phi(r,x,h)m_{n_k}(r;\ud x,\ud h)\\
                                        &\phantom{aaaa}-\int_{\bR\times\bR_+}D_pH_1(x,h,\bar{m}(r),D_hV(r,x,h))D_h\phi(r,x,h)\bar{m}(r;\ud x,\ud h)\\
\label{eq:int_FP_1}                             &\phantom{-}\leq\int_{\bR\times\bR_+}D_pH_1(x,h,m_{n_k}(r),D_hV_{n_k}(r,x,h))D_h\phi(r,x,h)\left(m_{n_k}(r;\ud x,\ud h)-\bar{m}(r;\ud x,\ud h)\right)\\
  \label{eq:int_FP_2}                             &\phantom{-\leq}+\int_{\bR\times\bR_+}\left(D_pH_1(x,h,m_{n_k}(r),D_hV_{n_k}(r,x,h))\right.\\
                                  &\phantom{-\leq\leq}\left.-D_pH_1(x,h,\bar{m}(r),D_hV(r,x,h))\right)D_h\phi(r,x,h)\bar{m}(r;\ud x,\ud h).
\end{align}
Let us look at (\ref{eq:int_FP_1}). The quantity $\sup_k\overline{M}(m_{n_k}(\cdot))$ is bounded by $K_1$ and $D_h\phi$ has compact support, so that thanks to Lemma~\ref{lem:DpH} there exists a constant $\hat{C}_\phi$ such that
\begin{multline*}
  \int_{\bR\times\bR_+}D_pH_1(x,h,m_{n_k}(r),D_hV_{n_k}(r,x,h))D_h\phi(r,x,h)\left(m_{n_k}(r;\ud x,\ud h)-\bar{m}(r;\ud x,\ud h)\right)\\ \leq \hat{C}_\phi\int_{\bR\times\bR_+}D_h\phi(r,x,h)\left(m_{n_k}(r;\ud x,\ud h)-\bar{m}(r;\ud x,\ud h)\right);
\end{multline*}
therefore (\ref{eq:int_FP_1}) converges to $0$. By Lemma~\ref{lem:DpH_2} we have that $D_pH_1(x,h,m_{n_k}(r),D_hV_{n_k}(r,x,h))-D_pH_1(x,h,\bar{m}(r),D_hV(r,x,h))$ converges to $0$ pointwise, while being uniformly bounded for the same reason as above; therefore by the dominated convergence theorem also (\ref{eq:int_FP_2}) converges to $0$.\\

\emph{Step 4.} By Schauder's fixed point theorem the map $\Psi_{\mu_0}$ has at least one fixed point $\mu(\cdot)\in\sQ_{K_1,K_2}$. Choose as $V$ the value function corresponding to such $\mu(\cdot)$; then the couple $(V,\mu)$ is a solution to (\ref{MFGb}) in the sense of Definition~\ref{def:sol}; this concludes the proof.
     \end{proof}

	\begin{remark}
	   By the regularity of the value function for each $\mu_n(\cdot)$, we know that $V_n$ is a solution of the Hamilton-Jacobi-Bellman equation in \eqref{eqn:HJB} where $H_1$ is evaluated in $\mu_{n}(\cdot)$. Then, by a similar procedure as the one in the proof of Theorem~\ref{thm:existence}, by Lemma \ref{lem:locgradV} and by the Ascoli-Arzel\`a theorem coupled with a diagonal argument, one could directly infer that $V_n$ converges locally uniformly to some $v$ and by stability of viscosity solutions applied to the HJB equation, one could then deduce that the subsequence $V_n$ converges to a solution $v$ of the same equation with $\tilde H_1$ evaluated at $\mu(\cdot)$. In any case, since no uniqueness of the solution of the HJB equation is guaranteed, it is not possible to conclude that the function $v$ obtained in this way is the value function corresponding to $\mu$. Therefore, due to how the operator $\Psi_{\mu_0}$ is defined, in the previous proof we need to prove directly the convergence of $V_n$ to $V$.
	\end{remark}

\bibliographystyle{abbrv} 
\bibliography{bibtex}

\end{document}